\documentclass[10pt]{amsart}

\usepackage{amsthm,amsmath,amssymb}
\usepackage{graphicx}
\usepackage[comma, sort&compress]{natbib}
\usepackage{tabularx}
\usepackage{multirow}
\usepackage{amssymb}
\usepackage{stackrel}
\usepackage{tikz}
\usetikzlibrary{shapes,arrows,positioning}

\usepackage{paralist}
\RequirePackage[hyperindex,breaklinks]{hyperref}

\newtheorem{theorem}{Theorem}

\newtheorem{lemma}[theorem]{Lemma}

\theoremstyle{definition}
\newtheorem{definition}{Definition}
\theoremstyle{remark}

\newtheorem{assumption}{Assumption}

\numberwithin{equation}{section}
\numberwithin{theorem}{section}
\numberwithin{example}{section}
\numberwithin{definition}{section}

\numberwithin{figure}{section}

\newcommand{\bE}{{\mathbb{E}}}

 \newcommand{\barPhi}{\bar{\Phi}}

\newcommand{\assumpref}[1]{Assumption~\ref{assump:#1}}
\newcommand{\assumpsref}[1]{Assumptions~\ref{assump:#1}}
\newcommand{\assumpssref}[1]{\ref{assump:#1}}

\newcommand{\secref}[1]{Section~\ref{sec:#1}}
\newcommand{\secsref}[1]{Sections~\ref{sec:#1}}
\newcommand{\secssref}[1]{\ref{sec:#1}}
\newcommand{\appref}[1]{Appendix~\ref{app:#1}}

\newcommand{\defref}[1]{Definition~\ref{def:#1}}

\newcommand{\lemref}[1]{Lemma~\ref{lem:#1}}

\newcommand{\thmref}[1]{Theorem~\ref{thm:#1}}

\newcommand{\tabref}[1]{Table~\ref{tab:#1}}
\newcommand{\tabsref}[1]{Tables~\ref{tab:#1}}
\newcommand{\tabssref}[1]{\ref{tab:#1}}

\title[]{Asymptotic false discovery control of the Benjamini-Hochberg procedure for pairwise comparisons}

\author[W.~Liu]{Weidong Liu}
\address{School of Mathematical Sciences, Shanghai Jiao Tong University}
\email{weidongl@sjtu.edu.cn}

\author[D.~Leung]{Dennis Leung}
\address{School of Mathematics and Statistics, University of Melbourne}
\email{dennis.leung@unimelb.edu.au}

\author[Q.~Shao]{Qi-Man Shao}
\address{Department of Statistics and Data Science, Southern University of Science and Technology}
\email{shaoqm@sustech.edu.cn}

\begin{document}

\begin{abstract}

%
%

In a one-way analysis-of-variance (ANOVA) model, the number of all pairwise comparisons can be large even when there are only a moderate number of groups. Motivated by this, we consider a regime with a growing number of groups, and prove that for testing pairwise comparisons the  BH procedure \citep{BH1995}  can offer asymptotic control on false discoveries, despite that the  t-statistics involved do not exhibit the well-known positive dependence structure called the PRDS to guarantee exact false discovery rate (FDR) control \citep{BY2001}. Sharing \citet{tukey1991philosophy}'s viewpoint that the difference in the means of any two groups cannot be exactly zero, our main result is stated in terms of the control on the directional false discovery rate and directional false discovery proportion \citep{BenAndDan2005}. A key technical contribution is that  we have shown the dependence among the t-statistics to be weak enough to induce a convergence result typically needed for establishing asymptotic FDR control. Our analysis does not adhere to stylized assumptions such as normality, variance homogeneity and a balanced design, and thus provides a theoretical grounding for applications in more general situations.

\end{abstract}



\maketitle

\section{Introduction} \label{sec:intro}

 Suppose we have $m$ independent groups of observations $\mathcal{X}_{i}=\{X_{ki},1\leq k\leq n_{i}\}$, $1 \leq i \leq m$, where for each $i$, $X_{ki}$'s are independent and identically distributed  random variables with mean $\mu_{i}$ and variance $\sigma_i^2$. The pairwise comparison problem
\begin{equation}\label{a0}
H_{ij}:~\mu_{i}=\mu_{j}\quad \mbox{against}\quad K_{ij}:~\mu_{i}\neq \mu_{j}, \ \ 1 \leq i < j \leq m
\end{equation}
has been widely studied since \citet{tukey1953PMC}'s early work on  multiple comparisons. In the early days, \citet{tukey1953PMC} and \citet{Kramer} independently proposed their famous procedure for testing \eqref{a0} based on the studentized range distribution, with the goal of  controlling the  family-wise error rate (FWER). Such developments culminated in the work of \citet{Hayter} who established the conservativeness of the Tukey-Kramer procedure in the affirmative. In modern applications when $m$ can be large,  the number of hypotheses $q:= {m \choose 2} = m(m-1)/2$ to consider is  even larger, making any testing procedure aiming to control FWER too conservative to be useful.  Hence there is a strong case for using the false discovery rate (FDR) proposed by \citet{BH1995}, defined as the expectation of the false discovery proportion (FDP),
\[
\frac{\# \text{hypotheses incorrectly rejected}}{\# \text{hypotheses rejected}},
\]
as a more appropriate type 1 error measure for pairwise comparisons,  given its scalability to the number of rejections (``discoveries") made.


The original step-up procedure proposed in \citet{BH1995} (widely known as the BH procedure, or BH for short) is proven to control the FDR at a pre-specified level $0 < \alpha < 1$ \emph{ when the test statistics are independent}.  However, the test statistics for  the pairwise comparison problem \eqref{a0}, namely, the two-sample t-statistics,
\begin{equation} \label{2samplet}
T_{ij}=\frac{\bar{X}_{i}-\bar{X}_{j}}{\sqrt{\hat{\sigma}^{2}_{i}/n_{i}+\hat{\sigma}^{2}_{j}/n_{j}}}, \ \ 1 \leq i \neq  j \leq m,
\end{equation}
where $\bar{X}_{i}=\frac{1}{n_{i}}\sum_{k=1}^{n_{i}}X_{ki}$ is the sample mean and $\hat{\sigma}^{2}_{i}=\frac{1}{n_{i}-1}\sum_{k=1}^{n_{i}}(X_{ki}-\bar{X}_{i})^{2}$ is the sample variance for group $i$, are apparently dependent in a specific pattern. If for each unique index pair $i < j$, we let
\begin{equation} \label{1sidepv}
p_{ij, U}=1- F_{ij}(T_{ij}) \text{ and } p_{ij, L}=F_{ij}(T_{ij}),
\end{equation}
be the one-sided upper and lower tailed p-values of $T_{ij}$,  where $F_{ij}$ is the (approximate) cumulative distribution function of $T_{ij}$ under $H_{ij}$,
the BH procedure would stipulate that we first sort
all the two-sided p-values,  defined as
\begin{equation} \label{2sidepv}
p_{|ij|} = 2 \min(p_{ij, L}, p_{ij, U}) \text{ for each pair } i< j,
\end{equation}
into their order statistics
\[
 p_{(1)}\leq p_{(2)}\leq\ldots\leq p_{(q)},
 \]
 and reject the null hypothesis $H_{ij}$ whenever $p_{|ij|}\leq p_{(\hat{k})}$ for $
\hat{k}:=\max\{i: ~p_{(i)}\leq i\alpha/q\}
$ (and if $\hat{k}$ is the maximum over an empty set then no rejection is made). Despite the dependence among the statistics in \eqref{2samplet}, many studies suggested that the BH procedure can still provide valid FDR control based on extensive simulations \citep*{williams, keselman1999pairwise, blair1995improved}.
%
The subsequent well-cited work of \citet{BY2001} has proven the validity of BH under ``positive dependence" of the test statistics. Unfortunately, this particular dependence  condition, widely known as PRDS in the literature, is not satisfied by the  two-sample t-statistics in \eqref{2samplet}; see  \citet[Example 2.6]{yekutieli2008false}. Indeed, even under the seemingly innocuous setup with a balanced design, as well as normality and homogeneity of error variances, establishing exact FDR control of the BH procedure for pairwise comparisons remains  a hard open problem, as
\citet[p.1882]{BY2001} also pointed out: ``\emph{Another important
open question is whether the same procedure controls the FDR when
testing pairwise comparisons of normal means, either Studentized or not}."

In real applications, the number of tests $q$ can become large for only a moderate number of groups $m$; for example, $(m,q)=(10,45)$ in \citet{oishi1998measurement},
 $(m,q)=(41,820)$ in  \citet{williams} and $(m,q)=(72,2556)$ in \citet{pawluk2006application}.
Motivated by this, we here pursue another avenue and show,  from an empirical-process viewpoint,  that the BH procedure is asymptotically valid for the pairwise comparison problem when $m$ tends to $\infty$ at a rate controlled by the available sample size. This approach is perhaps best represented by the FDR works of  \citet{storey2004strong} and \citet{genovese2004stochastic}, and a key ingredient in proving BH's asymptotic validity is the convergence of an underlying empirical process, which is  \emph{assumed} to hold under some general weak dependence among the test statistics \citep[p.193]{storey2004strong}.  In this respect,  one main offering  of our work is that we have shown
the dependence among the two-sample t-statistics in \eqref{2samplet} to be ``weak enough" to warrant such  convergence in the form of a uniform weak law of large numbers (\secref{ULLNpf}), which is not immediately obvious and also naturally leads to results on the asymptotic control of the FDP; compare \citet[Theorem 6]{storey2004strong}. Moreover, an asymptotic treatment allows us the flexibility to do away with standard assumptions such as normality, variance homogeneity and a balanced design, which can be easily violated in applications.

Originally,  \citet{tukey1991philosophy, tukey1962future} argued that, for pairwise comparisons,  the difference $\mu_i - \mu_{j}$ for any pair $(i, j)$  can never be exactly zero in reality and at best be
close to zero to many decimal points. As such a type 1 error of rejecting a true null hypothesis $H_{ij}$ can never occur. Nevertheless, one can still make sense of the BH procedure by making a slight modification: Following the rejection of   a hypothesis $H_{ij}$,
a declaration of  the sign of the difference $\mu_i - \mu_j$ must be made by the practitioner.  If 
$T_{ij} > 0$, then $\text{sgn}(\mu_i - \mu_j)$ will be declared as positive, and vice versa.
Making a wrong sign declaration will constitute what is known as a  \emph{directional error}.
Following more recent works such as \citet{benjamini2000adaptive} and \citet{BenAndDan2005},  we acknowledge the possibility that some null hypotheses in \eqref{a0} could in fact be true, and define a directional error more generally as follows:
\begin{definition} \label{def:type3}
A directional error for a hypothesis $H_{ij}$ is made if,  either  $\text{sgn}(\mu_i - \mu_j)$ is incorrectly declared after a rejection of $H_{ij}$ when $\mu_i - \mu_j \not= 0$, or  $H_{ij}$ is rejected when $\mu_i - \mu_j = 0$.
\end{definition}
Hence, a directional error  is equivalent to a usual type 1 error for $H_{ij}$ when $\mu_i - \mu_j = 0$.
Under \defref{type3}, we formally define \emph{the directional false discovery proportion} (dFDP)  as
\begin{equation} \label{newFDR}
dFDP = \frac{
\#\{ (i, j): \text{a directional error is made for $(i, j)$}\}
}{\#\{(i, j): H_{ij} \text{ is rejected} \}}
\end{equation}
and the directional false discovery rate (dFDR) as the expectation thereof. To accommodate Tukey's viewpoint, our main result (\thmref{main})  is thus stated in terms of dFDR and dFDP control. To the best of our knowledge, asymptotic results on dFDR/dFDP control have not appeared elsewhere in the literature. While our techniques can certainly be employed to prove a similar theorem in terms of the original FDR,  it is instructive to demonstrate how they can be adopted to prove a result for dFDR, using the pairwise comparison problem as a showcase.

\subsection{Organization and notation}
\secref{main} states and discusses our main result, which is immediately followed by its proof that relies on a key uniform weak law of large numbers proved later in \secref{ULLNpf}. \secsref{numeric} conducts a simple numerical study to shed light on our main result, while \secssref{conclude} concludes with a discussion of open issues. To fix notation,  we  let $n := \max_i n_i$. Throughout, $C, c$ are  positive constants whose values, unless otherwise specified,  are understood to not depend on $(m, n)$, and may vary from place to place.  The dependence of $m = m_n$ on $n$ is implicit.
The ``big-O,  little-o" notation is as usual: For two sequences of real numbers $\{a_n\}$ and $\{b_n\}$, $a_n = O(b_n)$ means that $|a_n/b_n|$ is bounded, and $a_n = o(b_n)$ means that $a_n/b_n \longrightarrow 0$ as $m,n \longrightarrow \infty$.
The notation $Card(A)$ denotes the cardinality of a set $A$. $\Phi(\cdot )$ is the cumulative distribution function of the standard normal and $\barPhi(\cdot) := 1 - \Phi(\cdot)$; $\phi(\cdot)$ is the standard normal density. $\log_d$ means taking $\log$ for $d$ times and for two functions $f$ and $g$, $f(x) \sim g(x)$, $x \rightarrow \infty$ means $f(x)/g(x) \longrightarrow 1$ as $x$ goes to infinity. For real numbers $a, b, \in \mathbb{R}$, $a\vee b$ and $a \wedge b$ are shorthand for $\max(a, b)$ and $\min(a, b)$;  if $a$ and $b$  are non-negative, $a \lesssim b$ means there
exists a constant $C$ not depending on $(m, n)$ such that $a \leq Cb$ holds. The Euclidean norm of a vector ${\bf x} \in \mathbb{R}^p$ is denoted as $\|{\bf x}\|$.

\section{Main results} \label{sec:main}

We first formally
define the version of the BH procedure accompanied by sign declaration upon the rejection of a null hypotheses in  \eqref{a0}.

\begin{definition}[Level-$\alpha$ BH with sign declaration for pairwise comparisons] \label{def:dBH}
\ \
\begin{enumerate}
\item Obtain the two-sided p-values $p_{|ij|}$ according to \eqref{2samplet}, \eqref{1sidepv} and \eqref{2sidepv}, and sort them as the order statistics  $ p_{(1)}\leq p_{(2)}\leq\ldots\leq p_{(q)}$, where $q = m(m-1)/2$.
\item Let $\hat{k}=\max\{k: ~p_{(k)}\leq k\alpha/q\}$, and reject $H_{ij}$ whenver $p_{|ij|} \leq p_{(\hat{k})}$. No rejection will be made if $\hat{k}$ cannot be properly defined .
\item For each rejected $H_{ij}$, declare $\text{sgn}(\mu_i - \mu_j) > 0 $ if $T_{ij} >0$, and declare $\text{sgn}(\mu_i - \mu_j) < 0 $ if $T_{ij}  <0$.
\end{enumerate}

\end{definition}

Note that the $F_{ij}$'s for forming the p-values do not have to be the exact null distribution functions of the $T_{ij}$'s. Our main result, \thmref{main} below, establishes the validity of the above BH procedure under the following assumptions:

\begin{assumption} [Approximate balance and homogeneity]
\label{assump:balance_homo}
There exist  $0 < c_L \leq 1 \leq c_U $ such that for all $1 \leq i \not = j \leq m$,
\[
c_L \leq  \sigma_i^2/\sigma_j^2 , \ \ n_i/n_j \leq  c_U.
\]
\end{assumption}

\begin{assumption} [ Asymptotic regime and moments]
 \label{assump:moment_regime}
Suppose that for some constants $r, C, \nu, K>0$,
\[
m \leq C n^r \text{ and } \max_{1 \leq i \leq m} \mathbb{E}|(X_{1i} - \mu_i)/\sigma_i |^{2 + 4r  + \nu} \leq K
\]
\end{assumption}

\begin{assumption}[Uniform Cram\'er-type moderate deviations of the reference null distributions]  
 \label{assump:ref_null}
 For all $i < j$, $F_{ij}$'s are monotone increasing distribution functions, with the symmetric properties that $F_{ij}(t) = 1 -  F_{ij}(- t)$ for $t \geq 0$ and
\[
\max_{1 \leq i < j \leq m} \sup_{0 \leq t \leq 2 \sqrt{\log m}    } \left|  \frac{1 - F_{ij}(t)}{\bar{\Phi}(t)} - 1\right| = O(n^{-c}) \text{ as } m, n \longrightarrow \infty
\]
for some $c>0$, where the possible dependence of the distribution functions $F_{ij}$ on $(m, n_1, \dots, n_m)$ is suppressed for notational simplicity.
\end{assumption}

These are very modest assumptions. \assumpref{balance_homo} allows for an imbalanced design and variance non-homogeneity in a controlled manner, while \assumpref{moment_regime} allows the growth of $m$ in $n$ to be of a polynomial order that  depends on the variables' moments. In particular, $4r + \nu$ can in fact be less than $1$, despite the commensurate limited growth rate of $m$ in that case. In \assumpref{ref_null}, the symmetry property of
a  reference null distribution functions $F_{ij}$ covers the usual normal and Student's t distributions. The same assumption also requires them to be approximately normal in a uniform way. This later condition trivially holds when all $F_{ij}$'s are taken to be $\Phi(\cdot)$. By the Cram\'er-type moderate deviation for  t-statistics \citep{Jing}, it will also hold when the $F_{ij}$'s are taken as  t-distribution functions with, say, $\min(n_i ,  n_j) - 1$ degrees of freedom \citep{scheffe}, under the approximately balanced design and asymptotic regime specified in \assumpsref{balance_homo} and \assumpssref{moment_regime}.

\begin{theorem}[Asymptotic FDR and FDP control for pairwise comparisons] \label{thm:main}
Let $\mathcal{H}_0 := \{(i, j): i < j, \ \ \mu_i - \mu_j = 0 \}$ be the set of null hypothesis indices and define $q_0 := |\mathcal{H}_0 |$. Under \assumpsref{balance_homo} to \assumpssref{ref_null} and  the condition
\begin{equation} \label{signalsize}
\text{Card}\left\{ (i, j) : 1 \leq i < j \leq m, \frac{ |\mu_i - \mu_j|}{\sqrt{\sigma^2_i/n_i + \sigma^2_j/ n_j}} \geq 8 c_U^2\sqrt{\log m}\right\} \geq 1,
\end{equation}  the testing procedure in \defref{dBH} has the properties that
\begin{equation} \label{FDRcontrol}
\limsup_{m,n } \frac{dFDR}{\frac{\alpha}{2} \left(1 + \frac{q_0}{q}\right)} \leq 1
\end{equation}
and for any $\epsilon > 0$,
\begin{equation}
\label{FDPcontrol}
P\left(dFDP < \frac{\alpha}{2} \left(1 + \frac{q_0}{q}\right) + \epsilon\right)\longrightarrow 1
\end{equation}
as $m, n \longrightarrow \infty$, where the $dFDP$ and $dFDR$ are defined as in \secref{intro}.
\end{theorem}

%

The dFDR bound of the form $\frac{\alpha}{2} (1 + \frac{q_0}{q})$ has appeared in the exact results established by \citet[Corollary 3 and Corollary 6]{BenAndDan2005} for test statistics that are independent or positively dependent. (More precisely, their bound takes the form $\frac{\alpha}{2} (1 + \frac{N_0}{N})$ where $N$ and $N_0$ respectively denote the number of hypotheses and the number of true nulls in a given multiple testing problem.) The term $ \frac{\alpha}{2} (1 + \frac{q_0}{q})$ becomes $\alpha/2$ when no hypotheses are true $(q_0 = 0)$. The latter is  intuitive  because in an extremely error-prone situation with all differences $\mu_i - \mu_j$ being very close to but not exactly zero, every $T_{ij}$ stochastically behaves almost as if $\mu_i - \mu_j = 0$, in which case upon rejecting $H_{ij}$ there is an approximate $1/2$ chance of making false declaration about $\text{sgn}(\mu_i  - \mu_j)$. In numerical studies where none  of the $\mu_i$'s are set to be equal, \citet{williams} showed that the dFDR can be controlled at $\alpha/2$ by the BH procedure in \defref{dBH} at level $\alpha$.

Lastly, \eqref{signalsize} is imposed for  the probabilistic control of the dFDP in \eqref{FDPcontrol}. In particular only one $(i, j)$ pair is required to give a more prominent signal (which will, however, necessarily imply that $O(m)$ many other pairs will also give similarly prominent signals; see \eqref{sigamt}). Ultimately, this is to ensure that with probability tending to 1, our BH procedure in \defref{dBH} will find a p-value cutoff  that isn't too close to zero; see displays \eqref{threshold} and \eqref{thresholdbdd}. The existence of such a non-zero cutoff is actually \emph{assumed} in the main asymptotic theorem of \citep[Theorem 4]{storey2004strong}, which represents the line of approach our current work is based on. In fact,  \citet[Proposition 2.1]{phase} has shown that even in the most ideal multiple testing setting with \emph{independent} and \emph{exact} $p$-values, the BH procedure cannot control the FDP with an overwhelming probability as the number of tested hypotheses increases if \emph{the number of  non-null hypotheses does not grow in tandem}; hence, \eqref{signalsize} is a kind of near necessary condition to guarantee asymptotic dFDP/FDP-type control. In a sense, our result complements a result of \citet[Corollary 3.5]{yekutieli2008false} which says that, under some conditions,  the BH procedure is valid for testing pairwise comparisons when the complete null hypothesis is true, i.e. $\mu_i = \mu_j$ for all pairs $i< j$.

We now give the main part of the proof for \thmref{main}; the weak convergence of a key  empirical process, which underpins this proof, will be treated separately in \secref{ULLNpf}.

\subsection{Proof of \thmref{main}}
It suffices to show \eqref{FDPcontrol}, since it implies that
\[
\limsup_{m, n}   \frac{\mathbb{E}[dFDP]}{\frac{\alpha}{2} (1 + \frac{q_0}{q})} \leq 1 + \epsilon
\]
 for any $\epsilon > 0$, and the arbitrariness of $\epsilon$ will give \eqref{FDRcontrol}. The following notation will be used throughout this section and the next: Let $\mathcal{H}_+ := \{(i, j): i < j, \ \ \mu_i - \mu_j > 0 \}$ and $\mathcal{H}_- := \{(i, j): i < j, \ \ \mu_i - \mu_j < 0 \}$. Define $q_+ := |\mathcal{H}_+| $ and $q_- = |\mathcal{H}_-| $, and hence $q = q_+ + q_- + q_0$. For notational brevity, we will also use $\sum_{\mathcal{H}_+}$ to denote a summation over all pairs $(i, j )$ in the set $\mathcal{H}_+$, and use $\sum_{\mathcal{H}_-}$, $\sum_{\mathcal{H}_0}$, $\sum_{\mathcal{H}_+ \cup \mathcal{H}_0}$  and $\sum_{\mathcal{H}_- \cup \mathcal{H}_0}$ similarly. Finally, for each pair $i \neq j$, we define the centered  two-sample t-statistic
\begin{equation} \label{centerTs}
\bar{T}_{ij} = \frac{(\bar{X}_{i}-\bar{X}_{j}) - (\mu_i - \mu_j)}{\sqrt{\hat{\sigma}^{2}_{i}/n_{i}+\hat{\sigma}^{2}_{j}/n_{j}}}.
\end{equation}

 Note that the rejection rule in \defref{dBH} $(ii)$ is equivalent to the classical $BH$ procedure, so by Theorem $2$ in \citet{BH1995} it is equivalent to a procedure that rejects $H_{ij}$ whenever $|p_{|ij|}| < \hat{\alpha}$, where
\[
\hat{\alpha} = \sup \left\{ 0 \leq \tilde{\alpha} \leq 1 :     \tilde{\alpha}\leq \frac{\alpha}{q} \sum_{i < j }{\bf 1}_{(p_{|ij|} \leq \tilde{\alpha})}\right\}.
\]
Since, with probability one,  $\frac{\alpha}{q} \sum_{i < j }{\bf 1}_{(p_{|ij|} \leq \tilde{\alpha})}$  is right continuous as a function in $\tilde{\alpha}$ on the interval $[0, 1]$, by elementary arguments it can be shown that
\begin{equation} \label{threshold}
\hat{\alpha} = \frac{\alpha}{q} \sum_{i < j }{\bf 1}(p_{|ij|} \leq \hat{\alpha}).
\end{equation}
If, for every $\epsilon > 0$, we can show that, as $m, n \longrightarrow \infty$,
\begin{multline} \label{ULLN}
P \left(\sup_{ 2 \bar{\Phi}(\sqrt{2 \log m})\leq \tilde{\alpha} \leq 1}\frac{\sum_{\mathcal{H}_+}  {\bf 1}_{( p_{ij, L} \leq \tilde{\alpha}/2)}+ \sum_{\mathcal{H}_-} {\bf 1}_{(  p_{ij, U} \leq \tilde{\alpha}/2)} + \sum_{\mathcal{H}_0} {\bf 1}_{(  p_{|ij|} \leq \tilde{\alpha})}}{\tilde{\alpha}(q + q_0)/2} \leq 1 +\epsilon \right) \\
\longrightarrow_p 1
\end{multline}
and
\begin{equation} \label{thresholdbdd}
P( 2 \bar{\Phi}(\sqrt{2 \log m}) \leq \hat{\alpha} \leq 1) \longrightarrow_p 1,
\end{equation}
then \eqref{FDPcontrol}  is proven since, by \eqref{threshold},  the expression
\begin{multline*}
 \frac{2q}{\alpha (q + q_0)} dFDP =\frac{ 2q[\sum_{\mathcal{H}_+}  {\bf 1}_{( p_{ij, L} \leq \hat{\alpha}/2)}+ \sum_{\mathcal{H}_-}  {\bf 1}_{(  p_{ij, U} \leq \hat{\alpha}/2)} + \sum_{\mathcal{H}_0}  {\bf 1}_{(  p_{|ij|} \leq \hat{\alpha})}]}{\alpha (q + q_0)\sum_{i < j }{\bf 1}_{(p_{|ij|} \leq \hat{\alpha})}}
\end{multline*}
is less than $1  + \epsilon$ with probability tending to one. \eqref{ULLN} essentially amounts to proving a uniform law of large number, which is deferred to \secref{ULLNpf}. We will focus on showing the probabilistic bound for the cutoff point $\hat{\alpha}$ in \eqref{thresholdbdd} for the rest of this section.

We claim that, under \assumpref{balance_homo}, there exists a subset $\mathcal{H}_s \subset \{(i, j): 1 \leq i < j \leq m\}$ such that,
\begin{equation} \label{sigamt}
\text{Card}\{\mathcal{H}_s\} =  \lfloor m/2 \rfloor\text{ and }
\frac{|\mu_i - \mu_j|}{\sqrt{ \sigma^2_i /n_i + \sigma^2_j /n_j}} \geq 4 \sqrt{\log m}  \text{ for each }(i, j) \in \mathcal{H}_s.
\end{equation}
By \eqref{signalsize}, without loss of generality we can assume
\begin{equation} \label{WLOG}
 \frac{ |\mu_1 - \mu_2|}{\sqrt{\sigma^2_1/n_1 + \sigma^2_2/ n_2}} \geq 8 c_U^2\sqrt{\log m}.
\end{equation}
There can only be two cases: If
$
\text{Card}\{ j : 3 \leq j \leq m , \frac{| \mu_1 - \mu_j|}{\sqrt{ \sigma^2_1 /n_1 + \sigma^2_j /n_j}}  \geq 4 \sqrt{\log m} \}\geq m/2,
$
then \eqref{sigamt} is satisfied; otherwise, we have
\begin{equation} \label{ow}
\text{Card}\left\{j : 3 \leq j \leq m , \frac{| \mu_1 - \mu_j|}{\sqrt{ \sigma^2_1 /n_1 + \sigma^2_j /n_j}}  < 4 \sqrt{\log m}\right\} \geq m/2.
\end{equation}
Since $|\mu_2 - \mu_j| \geq |\mu_2 - \mu_1| - |\mu_1 - \mu_j|$, by \assumpref{balance_homo},
\[
\frac{|\mu_2 - \mu_j|}{\sqrt{ \sigma^2_2/n_2+ \sigma^2_j/n_j}} \geq \frac{|\mu_2 - \mu_1|}{c_U\sqrt{\sigma^2_2/n_2+ \sigma^2_1/n_1   }} - \frac{c_U|\mu_j - \mu_1|}{\sqrt{\sigma^2_j/n_j+ \sigma^2_1/n_1   }}.
\]
Recalling that $ c_U \geq 1$, together with \eqref{WLOG} and \eqref{ow} we have
\[
\text{Card}\left\{ j : 3 \leq j \leq m ,  \frac{| \mu_2 - \mu_j|}{\sqrt{ \sigma^2_2 /n_2 + \sigma^2_j /n_j}}  \geq 4 \sqrt{\log m} \right\} \geq m/2
\]
and hence \eqref{sigamt} also holds.  The maximal inequality in the appendix (\lemref{max_ineq_variances}) also states that
\begin{equation*} \label{eq1}
P\left(\max_{1 \leq i \leq m}\left|\frac{\hat{\sigma}_i^2}{\sigma_i^2}  - 1 \right| > (\log m)^{-2}\right) = O(n^{-r - c}).
\end{equation*}
for some $c > 0$, which, together with  \eqref{sigamt},  gives
\begin{equation} \label{eq3}
P\left( \min_{(i, j) \in \mathcal{H}_s} \frac{|\mu_i - \mu_j|}{\sqrt{ \hat{\sigma}_i^2/n_i  +  \hat{\sigma}_j^2/n_j }} \geq 3.9 \sqrt{\log m}\right)  \longrightarrow 1.
\end{equation}

On the other hand, by the Cram\'er-type moderate deviation for two-sample t-statistics (\lemref{changzhou}), under \assumpsref{balance_homo} and \assumpssref{moment_regime} we have
\begin{equation} \label{2sampCramer}
\frac{P(\bar{T}_{ij}\geq s)}{\bar{\Phi}(s)} = 1+ O \Biggl[\left(\frac{1 + s}{n^{1/2 - 1/(4r + 2 + \nu)}}\right)^{4r + 2 + \nu}\Biggr]
\end{equation}
uniformly in $s \in [0, o(n^{1/2 - 1/(4r + 2 + \nu)}))$ and $1 \leq i < j \leq m$. A union bound then implies
\begin{multline} \label{eq4}
P\left(\max_{ (i, j)\in \mathcal{H}_s}\left| \bar{T}_{ij}\right| \geq \sqrt{2 \log m}\right) \leq  \lfloor m/2 \rfloor (2 + o(1)) \bar{\Phi}
(\sqrt{2 \log m}) \\
= O( (\log m)^{-1/2}) \longrightarrow 0
\end{multline}
using the fact that $\bar{\Phi}(x) \leq (2 \pi)^{-1/2}x^{-1} \exp(- x^2/2)$ for all $x > 0$.

In view of \eqref{eq3} and \eqref{eq4}, as well as $3.9 - \sqrt{2} > 2$, we have that
\[
P\left(\sum_{1 \leq i < j \leq m} {\bf 1}_{( |T_{ij}| \geq 2 \sqrt{\log m})} \geq \lfloor m/2 \rfloor \right) \longrightarrow 1.
\]
By \assumpref{ref_null}, $|T_{ij}| \geq 2 \sqrt{\log m}$ implies $p_{|ij|} = o(m^{-2})$, so recognizing $q = O(m^2)$ we have
\[
P\left( \sum_{1 \leq i < j \leq m}  {\bf 1}_{(p_{|ij|} \leq 1 /q )} \geq  \lfloor m/2 \rfloor\right) \longrightarrow 1,
\]
which in turn gives
\begin{equation} \label{critical_mass}
P\left(2 \bar{\Phi}(\sqrt{2 \log m}) \leq \frac{\alpha \sum_{1 \leq i < j \leq m}  {\bf 1}_{(p_{|ij|} \leq 2 \bar{\Phi}(\sqrt{2 \log m}))}}{q}\right) \longrightarrow 1
\end{equation}
since $2 \bar{\Phi}(\sqrt{2 \log m}) = O( (\log m)^{-1/2}m^{-1})$. By the definition of $\hat{\alpha}$ this implies \eqref{thresholdbdd}.

\ \

\ \

\ \

\section{Proof of the uniform law of large numbers} \label{sec:ULLNpf}

In this section, we will establish the following weak convergence of a key empirical process, which underpins our proof in the prior section:

\begin{lemma}[Uniform weak law of large numbers] \label{lem:keyWLLN}
Suppose \assumpsref{balance_homo} to \assumpssref{ref_null}  hold.  Then
\begin{equation}\label{ratio_cen}
\sup_{0 \leq t \leq \sqrt{2 \log m}}\left|\frac{\sum_{(i, j) \in \mathcal{H}_+ \cup \mathcal{H}_0}  {\bf 1}_{( \bar{\Phi}^{-1}(\bar{p}_{ij, L}) \geq t)}+ \sum_{ (i, j) \in \mathcal{H}_- \cup \mathcal{H}_0} {\bf 1}_{( \bar{\Phi}^{-1}(\bar{p}_{ij, U}) \geq t)} }{\bar{\Phi}(t)(q + q_0)} - 1 \right|\longrightarrow_p 0,
\end{equation}
where $\bar{p}_{ij, L} := F_{ij}(\bar{T}_{ij})$ and $\bar{p}_{ij, U} := 1-  F_{ij}(\bar{T}_{ij})$ are respectively one-sided lower and upper tailed p-values computed from the centered t-statistics in \eqref{centerTs}.
\end{lemma}

Note that this lemma automatically leads to  \eqref{ULLN}.
This can be easily seen to be true by observing that for each $(i, j)$, the events
\[
\{p_{ij, L} \leq \tilde{\alpha}/2\}, \quad \{ p_{ij, U} \leq \tilde{\alpha}/2\} \text{ and } \{p_{|ij|} \leq \tilde{\alpha}\}
\]
are identical to
\[
\{\bar{\Phi}^{-1}(p_{ij, L} )\geq \bar{\Phi}^{-1}(\tilde{\alpha}/2)\}, \{\bar{\Phi}^{-1}(p_{ij, U} )\geq \bar{\Phi}^{-1}(\tilde{\alpha}/2)\} \text{ and } \{\bar{\Phi}^{-1}(p_{|ij|} )\geq \bar{\Phi}^{-1}(\tilde{\alpha})\}
\]
respectively,
and   that $\bar{p}_{ij, L} \leq {p}_{ij, L}$ for $(i, j)\in \mathcal{H}_+ \cup \mathcal{H}_0$, $\bar{p}_{ij, U} \leq {p}_{ij, U}$ for $(i, j)\in \mathcal{H}_-\cup \mathcal{H}_0$.

Our proof of  \lemref{keyWLLN} essentially shows that for every threshold $t$, the  empirical process in question can be bounded by terms that converge to zero using Chebyshev's inequality. This strategy has also been adopted previously by  \citet{phase} to prove the FDP controlling property of the BH procedure under weak dependence among one-sample t-statistics. In  the current context of  pairwise comparisons with two-sample t-statistics, the key observation making this strategy possible is that, on rewriting the difference bounded  in  display \eqref{ratio_cen} as
\[
\frac{\sum_{(i, j) \in \mathcal{H}_+ \cup \mathcal{H}_0}  \left({\bf 1}_{( \bar{\Phi}^{-1}(\bar{p}_{ij, L}) \geq t)} - \bar{\Phi}(t)\right)+ \sum_{ (i, j) \in \mathcal{H}_- \cup \mathcal{H}_0}\left( {\bf 1}_{( \bar{\Phi}^{-1}(\bar{p}_{ij, U}) \geq t)}- \bar{\Phi}(t)\right)  }{\bar{\Phi}(t)(q + q_0)},
\]
upon expansion, the highest-order terms of its second moment are seen to involve pairs of index duples  that do not overlap. For example, one such term is
\[
\frac{1}{(q + q_0)^2}\sum_{\substack{(i, j) \in \mathcal{H}_0\cup \mathcal{H}_+\\(i', j')\in \mathcal{H}_0\cup \mathcal{H}_- \\ |\{i, j\} \cap \{i', j'\}|= 0} } 
\left( \frac{P(\bar{\Phi}^{-1}(\bar{p}_{ij, L}) \geq t)}{\bar{\Phi}(t)} - 1\right) \left( \frac{P(\bar{\Phi}^{-1}(\bar{p}_{i'j', U}) \geq t)}{\bar{\Phi}(t)}  - 1\right).
\]
Here, each summand takes the product form $\left( \frac{P(\bar{\Phi}^{-1}(\bar{p}_{ij, L}) \geq t)}{\bar{\Phi}(t)} - 1\right) \left( \frac{P(\bar{\Phi}^{-1}(\bar{p}_{i'j', U}) \geq t)}{\bar{\Phi}(t)}  - 1\right)$ due to the independence between $\bar{p}_{ij, L}$ and $\bar{p}_{i'j', U}$ for $|\{i, j\} \cap \{i', j'\}| = 0$.
Since the number of summands, which equals ${m \choose 2}{m - 2 \choose 2}$,  and the divisor $(q + q_0)^2$ are both of order $O(m^4)$, the whole term should approach zero, since  $\frac{P(\bar{\Phi}^{-1}(\bar{p}_{ij, L}) \geq t)}{\bar{\Phi}(t)} - 1$ can be expected to vanish in a uniform manner. The ensuing proof  formalizes this.

\subsection{Proof of \lemref{keyWLLN}}
Let $0 = t_0 <  t_1 < ... < t_g = \sqrt{2 \log m} $ be such that 
\[
t_l - t_{l-1} = v_m  \text{ for } 1  \leq l \leq g-1 \qquad \text{ and } \qquad t_g - t_{g-1} \leq v_m,
\] where 
\[
v_m = (\sqrt{2  \log m} \log_4 m)^{-1} = (t_g \log_4 m)^{-1}.
\]
By the mean value theorem, for some $c_l \in (t_{l-1}, t_l)$,
\begin{equation} 
\frac{\bar{\Phi}(t_{l})}{\bar{\Phi}(t_{l-1})} = \frac{\bar{\Phi}(t_{l-1}) - v_m \phi(c_l)}{\bar{\Phi}(t_{l-1})} =1 - \frac{v_m \phi(c_l)}{
\bar{\Phi}(t_{l-1})}  = 1 + o(1).
\end{equation}
Since 
\[
0 < \frac{v_m \phi(c_l)}{\bar{\Phi}(t_{l-1})} \leq \frac{v_m \phi(t_{l-1})}{\bar{\Phi}(t_{l-1})} = (\log_4 m)^{-1} \frac{\phi(t_{l-1})}{t_g\bar{\Phi}(t_{l-1}) }
\]
and  that, whenever, $m \geq 2$
\[
\left|\frac{\phi(t)}{t_g\bar{\Phi}(t)}\right|\leq C \text{ for some universal constant } C \text{ and all  }  t \in [0, t_g]
\]
(the latter fact comes from the  well-known fact that $\frac{\phi(t)}{t \bar{\Phi}(t)} \sim 1$ as $t \longrightarrow \infty$; see \citet[p.113]{victor2009self} for instance), we have that
\begin{equation} \label{unifo1}
\max_{1 \leq l \leq g}\left|\frac{\bar{\Phi}(t_l)}{\bar{\Phi}(t_{l-1})} - 1 \right| \rightarrow 0 \text{ as } m, n \longrightarrow \infty.
\end{equation}
%
%
For any $t \in [t_{l-1} , t_{l}]$, $l = 1, \dots, g$, note that
\begin{multline} \label{tele2}
\frac{{\bf 1}_{(\bar{\Phi}^{-1} (\bar{p}_{ij, L})\geq t_l)}}{\bar{\Phi}(t_{l-1})} \leq \frac{{\bf 1}_{(\bar{\Phi}^{-1} (\bar{p}_{ij, L})\geq t)}}{\bar{\Phi}(t)} \leq \frac{{\bf 1}_{(\bar{\Phi}^{-1} (\bar{p}_{ij, L})\geq t_{l-1})}}{\bar{\Phi}(t_{l})} \;\ \text{ and }\\
\frac{{\bf 1}_{(\bar{\Phi}^{-1} (\bar{p}_{ij, U})\geq t_l)}}{\bar{\Phi}(t_{l-1})} \leq \frac{{\bf 1}_{(\bar{\Phi}^{-1} (\bar{p}_{ij, U})\geq t)}}{\bar{\Phi}(t)} \leq \frac{{\bf 1}_{(\bar{\Phi}^{-1} (\bar{p}_{ij, U})\geq t_{l-1})}}{\bar{\Phi}(t_{l})}.
\end{multline}
Combining   \eqref{unifo1} and \eqref{tele2}  implies that, for proving  \lemref{keyWLLN}, it suffices to show
\begin{equation} \label{finalprobconv}
\max_{0 \leq l \leq g}\left| \frac{\sum_{ (i, j) \in  \mathcal{H}_+ \cup \mathcal{H}_0}  {\bf 1}_{( \bar{\Phi}^{-1}(\bar{p}_{ij, L}) \geq t_l)}+ \sum_{ (i, j) \in  \mathcal{H}_- \cup \mathcal{H}_0}{\bf 1}_{( \bar{\Phi}^{-1}(\bar{p}_{ij, U}) \geq  t_l)} }{\bar{\Phi}( t_l)(q + q_0)}  - 1\right| \longrightarrow_p 0;
\end{equation}
to this end, for a given $\epsilon >0$, we will bound the probability 
\[
P\left(\max_{0 \leq l \leq g}\left| \frac{\sum_{\mathcal{H}_+ \cup \mathcal{H}_0}  {\bf 1}_{( \bar{\Phi}^{-1}(\bar{p}_{ij, L}) \geq t_l)}+ \sum_{\mathcal{H}_- \cup \mathcal{H}_0}{\bf 1}_{( \bar{\Phi}^{-1}(\bar{p}_{ij, U}) \geq  t_l)} }{\bar{\Phi}( t_l)(q + q_0)}  - 1\right|  > \epsilon \right)
\]
for the rest of this section.

First, note that
\begin{multline}\label{break_into_great_and_less}
P\left(\max_{0 \leq l \leq g}\left| \frac{\sum_{\mathcal{H}_+ \cup \mathcal{H}_0}  {\bf 1}_{( \bar{\Phi}^{-1}(\bar{p}_{ij, L}) \geq t_l)}+ \sum_{\mathcal{H}_- \cup \mathcal{H}_0} {\bf 1}_{( \bar{\Phi}^{-1}(\bar{p}_{ij, U}) \geq  t_l)} }{\bar{\Phi}( t_l)(q + q_0)}  - 1\right| > \epsilon \right)
\\
\leq \sum_{0 \leq l \leq g}P\left( \frac{\sum_{\mathcal{H}_+ \cup \mathcal{H}_0}  {\bf 1}_{( \bar{\Phi}^{-1}(\bar{p}_{ij, L}) \geq t_l)}+ \sum_{\mathcal{H}_- \cup \mathcal{H}_0} {\bf 1}_{( \bar{\Phi}^{-1}(\bar{p}_{ij, U}) \geq  t_l)} }{\bar{\Phi}( t_l)(q + q_0)}  > 1 + \epsilon \right)  +
 \\
 \sum_{0 \leq l\leq g}P\left( \frac{\sum_{\mathcal{H}_+ \cup \mathcal{H}_0}  {\bf 1}_{( \bar{\Phi}^{-1}(\bar{p}_{ij, L}) \geq t_l)}+ \sum_{\mathcal{H}_- \cup \mathcal{H}_0} {\bf 1}_{( \bar{\Phi}^{-1}(\bar{p}_{ij, U}) \geq  t_l)} }{\bar{\Phi}( t_l)(q + q_0)}  <1 - \epsilon \right).
\end{multline}
For any $l = 1, \dots, g$, from \eqref{tele2} we know that
\begin{multline*}
\frac{{\bf 1}_{(\bar{\Phi}(\bar{p}_{ij, L}) \geq t_l)}}{\bar{\Phi}(t_l)} \leq \frac{{\bf 1}_{(\bar{\Phi}(\bar{p}_{ij, L}) \geq t)}}{\bar{\Phi}(t)} \frac{\bar{\Phi}(t_{l-1})}{\bar{\Phi} (t_l)} \text{ and }\\
 \frac{{\bf 1}_{(\bar{\Phi}(\bar{p}_{ij, U}) \geq t_l)}}{\bar{\Phi}(t_l)} \leq \frac{{\bf 1}_{(\bar{\Phi}(\bar{p}_{ij, U}) \geq t)}}{\bar{\Phi}(t)} \frac{\bar{\Phi}(t_{l-1})}{\bar{\Phi} (t_l)}  \text{, for any } t \in [t_{l-1}, t_l].
\end{multline*}
From this and  \eqref{unifo1},  we can deduce that, for a large enough $n$ (and $m$) and all $l = 1, \dots g$, 
\begin{multline*}
\frac{{\bf 1}_{(\bar{\Phi}(\bar{p}_{ij, L}) \geq t_l)}}{\bar{\Phi}(t_l)} \leq \frac{{\bf 1}_{(\bar{\Phi}(\bar{p}_{ij, L}) \geq t)}}{\bar{\Phi}(t)} \left(1 + \frac{\epsilon}{2 + \epsilon}\right) \text{ and }\\
 \frac{{\bf 1}_{(\bar{\Phi}(\bar{p}_{ij, U}) \geq t_l)}}{\bar{\Phi}(t_l)} \leq \frac{{\bf 1}_{(\bar{\Phi}(\bar{p}_{ij, U}) \geq t)}}{\bar{\Phi}(t)}\left(1 + \frac{\epsilon}{2 + \epsilon}\right) \text{ for any } t \in [t_{l-1}, t_l],
\end{multline*}
which implies that 
\begin{multline}\label{1tog}
\sum_{1 \leq l \leq g}P\left( \frac{\sum_{\mathcal{H}_+ \cup \mathcal{H}_0}  {\bf 1}_{( \bar{\Phi}^{-1}(\bar{p}_{ij, L}) \geq t_l)}+ \sum_{\mathcal{H}_- \cup \mathcal{H}_0} {\bf 1}_{( \bar{\Phi}^{-1}(\bar{p}_{ij, U}) \geq  t_l)} }{\bar{\Phi}( t_l)(q + q_0)}  > 1 + \epsilon \right)  \\
\leq  v_m^{-1}\int_0^{t_g}P\left( \frac{\sum_{\mathcal{H}_+ \cup \mathcal{H}_0}  {\bf 1}_{( \bar{\Phi}^{-1}(\bar{p}_{ij, L}) \geq t)}+ \sum_{\mathcal{H}_- \cup \mathcal{H}_0} {\bf 1}_{( \bar{\Phi}^{-1}(\bar{p}_{ij, U}) \geq  t)} }{\bar{\Phi}( t)(q + q_0)}  > 1 + \epsilon/2 \right) dt,\\
\text{ for large enough } n \;(\text{and } m).
\end{multline}
With a completely analogous argument, one can then deduce that
%
%
\begin{multline} \label{0tog-1}
\sum_{0 \leq l\leq g-1}P\left( \frac{\sum_{\mathcal{H}_+ \cup \mathcal{H}_0}  {\bf 1}_{( \bar{\Phi}^{-1}(\bar{p}_{ij, L}) \geq t_l)}+ \sum_{\mathcal{H}_- \cup \mathcal{H}_0} {\bf 1}_{( \bar{\Phi}^{-1}(\bar{p}_{ij, U}) \geq  t_l)} }{\bar{\Phi}( t_l)(q + q_0)}  < 1 - \epsilon \right)  \\
\leq  v_m^{-1}\int_0^{t_g}P\left( \frac{\sum_{\mathcal{H}_+ \cup \mathcal{H}_0}  {\bf 1}_{( \bar{\Phi}^{-1}(\bar{p}_{ij, L}) \geq t)}+ \sum_{\mathcal{H}_- \cup \mathcal{H}_0} {\bf 1}_{( \bar{\Phi}^{-1}(\bar{p}_{ij, U}) \geq  t)} }{\bar{\Phi}( t)(q + q_0)}  < 1 - \epsilon/2 \right) dt,\\
\text{ for large enough } n \;(\text{and } m).
\end{multline}

Combining \eqref{break_into_great_and_less}, \eqref{1tog} and \eqref{0tog-1}, 
\begin{multline} \label{intandsum}
P\left(\max_{0 \leq l \leq g}\left| \frac{\sum_{\mathcal{H}_+ \cup \mathcal{H}_0}  {\bf 1}_{( \bar{\Phi}^{-1}(\bar{p}_{ij, L}) \geq t_l)}+ \sum_{\mathcal{H}_- \cup \mathcal{H}_0}{\bf 1}_{( \bar{\Phi}^{-1}(\bar{p}_{ij, U}) \geq  t_l)} }{\bar{\Phi}( t_l)(q + q_0)}  - 1\right|  > \epsilon \right)  \leq \\
v_m^{-1}\int_0^{t_g}P\left( \left| \frac{\sum_{\mathcal{H}_+ \cup \mathcal{H}_0}  {\bf 1}_{( \bar{\Phi}^{-1}(\bar{p}_{ij, L}) \geq t)}+ \sum_{\mathcal{H}_- \cup \mathcal{H}_0}{\bf 1}_{( \bar{\Phi}^{-1}(\bar{p}_{ij, U}) \geq  t)} }{\bar{\Phi}( t)(q + q_0)}  - 1\right|  >  \epsilon/2 \right)  dt+ \\
\sum_{l = 0, g}P\left(\left| \frac{\sum_{\mathcal{H}_+ \cup \mathcal{H}_0}  {\bf 1}_{( \bar{\Phi}^{-1}(\bar{p}_{ij, L}) \geq t_l)}+ \sum_{\mathcal{H}_- \cup \mathcal{H}_0}{\bf 1}_{( \bar{\Phi}^{-1}(\bar{p}_{ij, U}) \geq  t_l)} }{\bar{\Phi}( t_l)(q + q_0)}  - 1\right|  > \epsilon \right),\\
\text{ for large enough } n \; ( \text{and } m).
\end{multline}
In fact for any  $t \in [0 , \sqrt{2 \log m}]$, for large enough $n$ the probabilities of the type in \eqref{intandsum} above satisfy
\begin{multline} \label{switchbk}
 P\left(\left| \frac{\sum_{\mathcal{H}_+ \cup \mathcal{H}_0}  {\bf 1}_{( \bar{\Phi}^{-1}(\bar{p}_{ij, L}) \geq t)}+ \sum_{\mathcal{H}_- \cup \mathcal{H}_0}{\bf 1}_{( \bar{\Phi}^{-1}(\bar{p}_{ij, U}) \geq  t)} }{\bar{\Phi}( t)(q + q_0)}  - 1\right| > \epsilon \right) =  \\
 P\left(\left|  \sum_{\mathcal{H}_+ \cup \mathcal{H}_0 } \Bigl({\bf 1}_{( - \bar{T}_{ij} \geq t_{ij})}- \barPhi(t) \Bigr)+ \sum_{\mathcal{H}_- \cup \mathcal{H}_0}\Bigl( {\bf 1}_{(  \bar{T}_{ij} \geq t_{ij})} - \barPhi(t) \Bigr)  \right| > \barPhi( t)(q + q_0) \epsilon \right)
\end{multline}
where
\[t_{ij} =  \bar{\Phi}^{-1}( \bar{\Phi}(t)(1 + \epsilon_{|\bar{T}_{ij}|})),
\]
for a number $ \epsilon_{|\bar{T}_{ij}|}$ that  depends on the absolute value $|\bar{T}_{ij}|$ but has the property that
\begin{equation}\label{eptij}
\Bigl|\epsilon_{|\bar{T}_{ij}|} \Bigr|\leq \epsilon_n \text{ for all } (i, j) \quad a.s., \quad \text{ for some deterministic sequence } \epsilon_n = O(n^{-c}).
\end{equation}
The argument leading to \eqref{switchbk} is a bit delicate and is deferred to \appref{delicate}. Continuing from \eqref{switchbk}, with Chebyshev's inequality, 
 we get that for $t \in [0 , \sqrt{2 \log m}]$ and suffciently large $n$,
 \begin{multline}
 P\left(\left| \frac{\sum_{\mathcal{H}_+ \cup \mathcal{H}_0}  {\bf 1}_{( \bar{\Phi}^{-1}(\bar{p}_{ij, L}) \geq t)}+ \sum_{\mathcal{H}_- \cup \mathcal{H}_0}{\bf 1}_{( \bar{\Phi}^{-1}(\bar{p}_{ij, U}) \geq  t)} }{\bar{\Phi}( t)(q + q_0)}  - 1\right| > \epsilon \right) \leq \Bigl(\epsilon \bar{\Phi}(t)(q + q_0)\Bigr)^{-2} \\ 
 \times \Biggl\{  
 \sum_{\substack{(i, j) \in \mathcal{H}_0\cup \mathcal{H}_+\\(i', j')\in \mathcal{H}_0\cup \mathcal{H}_+ \\ |\{i, j\} \cap \{i', j'\}|= \ell\\ \ell = 0, 1, 2} } 
 \Bigl\{P(- \bar{T}_{ij} \geq t_{ij}, - \bar{T}_{i'j'} \geq t_{i'j'}) + \barPhi(t)^2 - \barPhi(t) [P(- \bar{T}_{ij} \geq t_{ij}) + P( - \bar{T}_{i'j'} \geq t_{i'j'})] \Bigr\}\\
 2 \sum_{\substack{(i, j) \in \mathcal{H}_0\cup \mathcal{H}_+\\(i', j')\in \mathcal{H}_0\cup \mathcal{H}_- \\ |\{i, j\} \cap \{i', j'\}|= \ell\\ \ell = 0, 1, 2} } 
 \Bigl\{P(- \bar{T}_{ij} \geq t_{ij},  \bar{T}_{i'j'} \geq t_{i'j'}) + \barPhi(t)^2 - \barPhi(t) [P(- \bar{T}_{ij} \geq t_{ij}) + P(  \bar{T}_{i'j'} \geq t_{i'j'})] \Bigr\}\\
  \sum_{\substack{(i, j) \in \mathcal{H}_0\cup \mathcal{H}_-\\(i', j')\in \mathcal{H}_0\cup \mathcal{H}_- \\ |\{i, j\} \cap \{i', j'\}|= \ell\\ \ell = 0, 1, 2} } 
 \Bigl\{P( \bar{T}_{ij} \geq t_{ij},  \bar{T}_{i'j'} \geq t_{i'j'}) + \barPhi(t)^2 - \barPhi(t) [P( \bar{T}_{ij} \geq t_{ij}) + P(  \bar{T}_{i'j'} \geq t_{i'j'})] \Bigr\}
 \Biggr\}. 
    \label{bigexpansion} 
 \end{multline}
Note that, for each of the three sums appearing on the right hand side, the combinations over pairs of index duples $\{i, j\}$ and $\{i', j'\}$ overlapping for 0, 1 and 2 elements respectively give rise to $O(m^4)$, $O(m^3)$ and $O(m^2)$ many summands. To finish the proof we need the following lemma whose proof is given in \appref{pfBerman}, where we define
\[
T_{ij}^* = \frac{ (\bar{X}_i - \mu_i) -( \bar{X}_j - \mu_j)}{\sqrt{\sigma^2_i/n_i + \sigma^2_j/n_j}}.
\]
\begin{lemma} \label{lem:BermanBdd}
Under our \assumpsref{balance_homo} to \assumpssref{ref_null}, we have, for large enough $n$,
\begin{enumerate}
\item \[
\frac{P(\bar{T}_{ij}\geq t_{ij})}{\bar{\Phi}(t)} =1 +O(n^{-c})
\]
uniformly both in  $ 0 \leq t \leq \sqrt{2 \log m}$ and in  all $i \neq j$, and
\item for some constant $\delta \in (0, 1)$,
\begin{align*}
&P\Bigl(T^*_{i_1j_1} \geq (1 - (\log m)^{-2})t_{i_1j_1}, T^*_{i_2 j_2} \geq (1 - (\log m)^{-2})t_{i_2 j_2} \Bigr) \\
&\leq C (1+ t)^{-2}\exp (-t^2 / (1+ \delta)) + O(n^{-r - c}),
\end{align*}
uniformly both in $0 \leq t \leq   \sqrt{2\log m}$ and in  all pairs of duples $(i_1, j_1), (i_2, j_2)$ such that $|\{i_1, j_1\} \cap \{i_2, j_2\}|  = 1$, $i_1 \neq j_1$ and $i_2 \neq j_2$.

\end{enumerate}
\end{lemma}

 By applying the lemma to the right hand side of \eqref{bigexpansion}, we continue to get that  for $t \in [0 , \sqrt{2 \log m}]$ and suffciently large $n$,
 \begin{align}
 & P\left(\left| \frac{\sum_{\mathcal{H}_+ \cup \mathcal{H}_0}  {\bf 1}_{( \bar{\Phi}^{-1}(\bar{p}_{ij, L}) \geq t)}+ \sum_{\mathcal{H}_- \cup \mathcal{H}_0}{\bf 1}_{( \bar{\Phi}^{-1}(\bar{p}_{ij, U}) \geq  t)} }{\bar{\Phi}( t)(q + q_0)}  - 1\right| > \epsilon \right)\notag\\
  & \stackrel{(a)}{\lesssim } \Biggl\{   \sum_{\substack{(i, j) \in \mathcal{H}_0\cup \mathcal{H}_+ \\(i', j')\in \mathcal{H}_0\cup \mathcal{H}_- \\ |\{i, j\} \cap \{i', j'\}|= 1} } 
   P( - \bar{T}_{ij} \geq t_{ij}, \bar{T}_{i'j'} \geq t_{i'j'}) +  \sum_{\substack{(i, j) \in \mathcal{H}_0\cup \mathcal{H}_+ \\(i', j')\in \mathcal{H}_0\cup \mathcal{H}_+ \\ |\{i, j\} \cap \{i', j'\}|= 1} } 
   P( - \bar{T}_{ij} \geq t_{ij}, -\bar{T}_{i'j'} \geq t_{i'j'}) +  \notag\\
   &  \sum_{\substack{(i, j) \in \mathcal{H}_0\cup \mathcal{H}_- \\(i', j')\in \mathcal{H}_0\cup \mathcal{H}_- \\ |\{i, j\} \cap \{i', j'\}|= 1} } 
   P(  \bar{T}_{ij} \geq t_{ij}, \bar{T}_{i'j'} \geq t_{i'j'}) 
  \qquad \Biggr\} \Bigl(\epsilon \bar{\Phi}(t)(q + q_0)\Bigr)^{-2}  + O\left(\frac{1}{m} + \frac{1}{\bar{\Phi}(t)m^2} + \frac{1}{n^c}\right) \notag\\
   &\stackrel{(b)}{\lesssim} \Biggl\{   \sum_{\substack{(i, j) \in \mathcal{H}_0\cup \mathcal{H}_+ \\(i', j')\in \mathcal{H}_0\cup \mathcal{H}_- \\ |\{i, j\} \cap \{i', j'\}|= 1} } 
   P( - T^*_{ij} \geq(1 - (\log m)^{-2}) t_{ij}, T^*_{i'j'} \geq (1 - (\log m)^{-2}) t_{i'j'})   \Biggr\} \Bigl(\epsilon \bar{\Phi}(t)(q + q_0)\Bigr)^{-2} \notag\\
   &\quad +  O\left(\frac{1}{m} + \frac{1}{\bar{\Phi}(t)m^2} + \frac{1}{n^c}  + \frac{1}{m \bar{\Phi}^2(x)n^{c+r}} \right)      \notag\\
&  \stackrel{(c)}{\lesssim} 
 O\left(\frac{\exp (-t^2 / (1+ \delta)}{\bar{\Phi}^2(t) m(1+ t)^2}\right) + 
 O\left(\frac{1}{\bar{\Phi}^2 (t) m n^{c+r}}\right) +
O\left(\frac{1}{m} + \frac{1}{\bar{\Phi}(t)m^2} + \frac{1}{n^c}  + \frac{1}{m \bar{\Phi}^2(x)n^{c+r}} \right) \notag\\
&  \stackrel{(d)}{=}  O\left( \frac{\exp(\delta (1 + \delta)^{-1}t^2)}{m} \right) + O\left( \frac{(1+t)^2 \exp(t^2)}{m n^{c+r}}\right) + 
O\left(\frac{1}{m} + \frac{(1+t) \exp(t^2/2)}{m^2} + \frac{1}{n^c}  + \frac{(1+t)^2 \exp(t^2)}{m n^{c+r}} \right)
\notag\\
&\stackrel{(e)}{=}  O\Biggl( m^{\frac{\delta -1}{1 + \delta}} + \frac{(1 + \sqrt{2 \log m})^2m}{n^{c+r}} + \frac{1}{m} + \frac{1 +\sqrt{ 2 \log m}}{m} + \frac{1}{n^c}
+ \frac{(1 + \sqrt{2 \log m})^2 m }{n^{c+r}} \Biggr)\label{lastlastOfcomplicated}
\end{align}

We now explain each of the annotated (in)equalities above:
\begin{enumerate}[$(a)$]
\item It results from the applications of \lemref{BermanBdd}$(i)$ to the right hand side of \eqref{bigexpansion}. and recognizing that $q + q_0$ is of the same order  as $m^2$, 
\item From the maximal inequality for the sample variances (\lemref{max_ineq_variances}), we have that with probability of at least $1 - O(n^{-r-c})$, \[
\sqrt{\frac{\hat{\sigma}_i^2/n_i + \hat{\sigma}_j^2/n_j}{\sigma_i^2/n_i + \sigma_j^2/n_j}}
= \sqrt{1 - \frac{( \sigma_i^2 - \hat{\sigma}_i^2)/n_i + ( \sigma_j^2 - \hat{\sigma}_j^2)/n_j}{\sigma_i^2/n_i + \sigma_j^2/n_j}}
 \geq 
\sqrt{ 1 - (\log m)^{-2}}
\]
for all $(i, j)$, which also implies that
\begin{multline*}
   P( - \bar{T}_{ij} \geq t_{ij}, \bar{T}_{i'j'} \geq t_{i'j'}) \leq \\
    P( - T^*_{ij} \geq(1 - (\log m)^{-2}) t_{ij}, T^*_{i'j'} \geq (1 - (\log m)^{-2}) t_{i'j'}) + O(n^{-r-c}).
\end{multline*}
To save space we have only shown the summands satisfying $(i, j) \in \mathcal{H}_0 \cup \mathcal{H}_+$ and  $(i', j') \in \mathcal{H}_0 \cup \mathcal{H}_-$, where the other summands are absorbed into ``$\lesssim$".
\item This results from the application of \lemref{BermanBdd}$(ii)$.

\item We used the fact that $(1 + t)^{-1} \phi(t)\lesssim \bar{\Phi}(t)$.

\item Substituting $\sqrt{2\log m}$ for $t$.
\end{enumerate}

Collecting \eqref{intandsum}, \eqref{switchbk} and \eqref{lastlastOfcomplicated}, we have proved \eqref{finalprobconv} under the asymptotic regime in \assumpref{moment_regime}.

\section{Numerical studies} \label{sec:numeric}

We conducted a numeric study  to shed light on our main theoretical result (\thmref{main}). For simplicity, we only consider balanced and homogenous setups where  $n_1 = \dots = n_m = n$ and $\sigma_1 = \dots = \sigma_m$. We generate $m$ population means $\mu_1, \dots, \mu_m$ such that the first $\mu_1, \dots, \mu_{m_0}$ are set to be zero, with $m_0 \leq m$ picked in such a way that 
\[
m_0  = \max \left\{ m' : 1 \leq m' \leq m\text{ and }{m' \choose 2} \leq \beta {m \choose 2} \right\},
\]
for some $\beta \in [0, 1)$, and the other $\mu_{m_0 +1} \dots \mu_m$ are
 i.i.d. realizations of a mean-zero normal distribution 
\[
N(\text{mean} =0, \text{sd} = \emph{effect size}),
\]
 where \emph{effect size} is a chosen value for the standard deviation, named as such since the larger \emph{effect size} is,  the larger are the magnitudes of the pairwise differences $\mu_i - \mu_j$, $ m_0+1 \leq i < j \leq m$; note that $\beta$ is roughly the same as $q_0/q$, the proportion of true nulls. Then for each $1 \leq i \leq m$, we generate $n$ i.i.d. data
 \[
 X_{ki} = \mu_i +  \epsilon_{ki}, \quad 1 \leq k \leq n,
 \]
 where $\epsilon_{k i}$ are independent 
  t-distributed error terms  with $12$ degrees of freedom, and apply the BH procedure in \defref{dBH} at level $\alpha$  using ${\Phi}$ as the reference distribution function to calibrate the p-values. For a given set of  $\mu_1, \dots, \mu_m$, the experiment is repeated 500 times to empirically estimate the performance of the BH procedure.  
At $\alpha = 0.2$, the empirical probabilities of the dFDP meeting (i.e. being less than) the desired target $\frac{\alpha}{2}(1+ \frac{q_0}{q})$ and the empirical dFDR's are reported  in \tabsref{FDP_beta_0_alpha_0.2} to \tabssref{FDR_beta_0.5_alpha_0.2} for $\beta = 0, 0.25, 0.5$, as well as for different combinations of $(m, n,  \emph{effect size})$; additional results for $\alpha = 0.1$ and $\alpha = 0.3$ are shown in  \appref{additional_numerical}.

Note that from the proof of \thmref{main} in \secref{main}, it is seen that the condition \eqref{signalsize} is a kind of minimalist assumption to ensure that, \emph{asymptotically}, the small two-sided p-values are prevalent enough so that the event 
 \begin{equation} \label{critical_event}
 2 \bar{\Phi}(\sqrt{2 \log m}) \leq \frac{\alpha \sum_{1 \leq i < j \leq m}  {\bf 1}_{(p_{|ij|} \leq 2 \bar{\Phi}(\sqrt{2 \log m}))}}{q}
 \end{equation}
happens with an overwhelming probability so that the p-value cutoff $\hat{\alpha}$ is not too small; revisit \eqref{thresholdbdd} and \eqref{critical_mass}. For finite $m$ and $n$, the event in \eqref{critical_event} is only true when there are enough small $p$-values, so one may generally expect better dFDP control in a signal rich environment. This is borne out  by  \tabsref{FDP_beta_0_alpha_0.2}, \tabssref{FDP_beta_0.25_alpha_0.2} and \tabssref{FDP_beta_0.5_alpha_0.2}, where it is seen that as $\beta$ increases, i.e. more true nulls are present in the system, the empirical probabilities 
$P(dFDP \leq \frac{\alpha}{2}\left( 1 + q_0/q\right))$ become progressively lower than $1$.  Moreover, for each of \tabsref{FDP_beta_0_alpha_0.2}, \tabssref{FDP_beta_0.25_alpha_0.2} and \tabssref{FDP_beta_0.5_alpha_0.2}, the larger the $\emph{effect size}$, the larger are these probabilities generally tending to become, which also provide further evidence for the near necessity of  a condition like \eqref{signalsize}.
The empirical probabilities in \tabref{FDP_beta_0.5_alpha_0.2} for \emph{effectsize} $= 0.05$ violate the latter trend somewhat,  as they tend to be larger than the same numbers for \emph{effectsize} $= 0.25$ when $m$ or $n$ are smaller. However, since \eqref{FDPcontrol} in \thmref{main}  is an asymptotic statement for large $m, n$, numerical results for finite $m, n$ should be interpreted with caution.

Lastly, except when $n$ is too small compared to $m$, the dFDR targets are met as seen in \tabsref{FDR_beta_0_alpha_0.2}, \tabssref{FDR_beta_0.25_alpha_0.2} and \tabssref{FDR_beta_0.5_alpha_0.2}.


\section{Discussion} \label{sec:conclude}

In summary, our work has established the validity of the BH procedure for pairwise comparisons in a flexible asymptotic framework suitable for a relatively large number of groups which does not require a strictly balanced design and variance homogeneity (\assumpsref{balance_homo}), by demonstrating that the dependence among all the two-sample t-statistics is weak enough to induce a requisite uniform weak law of large numbers. Using the Cram\'er-type moderate deviation as a tool, our result is present under  minimal moment assumptions, and hence the growth rate of the number of groups is polynomial in the sample size  (\assumpref{moment_regime}), as opposed to exponentially in the sample size had sub-Gaussian tails been assumed like in the related work of  \citet{phase} focusing on one-sample t-statistics. On a related note, while we have assumed that the $p$-values are calibrated  with deterministic reference distributions such as the standard normal (\assumpref{ref_null}), it is also possible to establish our results for calibration with bootstrap distributions, which has the potential to better the approximation accuracy \citep{fan2007many,delaigle2011robustness}. However, as demonstrated in the numerical studies of \citet{phase}, for heavy-tailed situations, bootstrap calibration has limited advantage in practice; we decided not to pursue this embellishment for simplicity.

While the BH procedure serves as a benchmarking procedure,  in recent years, other multiple-testing methods for controlling FDR-related quantities are  in active development.  It was brought to our attention by the Associate Editor that, to account for the dependence among the statistics for FDR testing, a recent series of works by \citet{zhou2018new, fan2019farmtest} and other references therein impose an \emph{approximate} factor model on the covariance structure between the test statistics, i.e. the covariance matrix is assumed to be the sum of a low rank  and an approximately diagonal \emph{uniqueness} matrix. In a nutshell, their proposal is  to subtract away the common factor structure among the statistics from the data prior to forming the p-values for the BH, with the goal of offering better false discovery control and improving the testing efficiency.  The pair comparison \eqref{a0} does have a natural factor structure: Let ${\bf X} = (X_1, \dots, X_m) =_d (X_{k1}, \dots, X_{km})^T$, i.e.  a random vector having the same distribution as $(X_{k1}, \dots, X_{km})^T$ whose corresponding mean is ${\boldsymbol \mu}_{\bf X} = (\mu_1, \dots, \mu_m)^T$, and define the ${m \choose 2}$-vector
\begin{equation} 
{\bf Y} = (\underbrace{X_1 - X_2, \dots, X_1 - X_m}_{ m - 1 \text{ times }}, \underbrace{X_2 - X_3, \dots, X_2 - X_m}_{ m - 2 \text{ times }}, \dots, , \underbrace{X_{m-1} - X_m}_{1 \text{ times}})^T.
\end{equation}
One can  then write the factor representation
\begin{equation} \label{Y_factor_struct}
{\bf Y} = {\boldsymbol \mu}_{\bf Y} + {\bf L} {\bf f},
\end{equation}
where
 \[
{\boldsymbol \mu}_{\bf Y} := (\underbrace{\mu_1 - \mu_2, \dots, \mu_1 - \mu_m}_{ m - 1 \text{ times }}, \underbrace{\mu_2 - \mu_3, \dots, \mu_2 - \mu_m}_{ m - 2 \text{ times }}, \dots, , \underbrace{\mu_{m-1} - \mu_m}_{1 \text{ times}})^T.
\]
captures the mean differences in \eqref{a0}, 
$
{\bf f} := {\bf X} - {\boldsymbol \mu}_{\bf X}
$
serves as an \emph{unobserved} mean-zero hidden factor with $m$ elements, and ${\bf L} = (L_{ij, l})$ is a ${m \choose 2} \times m$ loading matrix such that 
\[
 L_{ij, l} = 
  \begin{cases} 
   1       & \text{if } l = i\\
     - 1       & \text{if }  l = j\\
      0  & \text{if }  \text{ otherwise} 
  \end{cases}
\]
for $1 \leq i < j \leq m$ and $1 \leq l \leq m$. However, in the oracle case where ${\bf L}$ and ${\bf f}$ are hypothetically assumed to be observed, there is conceptual difficulty of applying \citet{zhou2018new, fan2019farmtest}'s approach  to test the mean differences in  ${\boldsymbol \mu}_{\bf Y} $, because upon subtracting away the factor structure, the residual ${\bf Y} - {\bf L} {\bf f}$ is a degenerate  vector  with no randomness to form the p-values, due to the lack of the usual idiosyncratic errors accounting for the uniqueness variances in the factor representation \eqref{Y_factor_struct}. In unreported numerical studies, we have also experimented  with the practical version of their approach where the underlying factors and loadings are estimated from the data ${\bf Y}$ (and hence not necessarily coinciding with ${\bf L} {\bf f}$), but the resulting false discovery control is very unstable; this is most possibly due to the fact that the underlying theory of \citet[Assumption $1(iv)$]{fan2019farmtest} still relies upon a non-degenerate uniqueness component in the  covariance structure.

%
%

In this paper, we treated the classical setup with only a single one-way ANOVA model. By comparison, one typical modern application  in genomics involves an experimental design where the measured expression levels of thousands, or even tens of thousands, of genes under several treatment groups are described by a multitude of one-way ANOVA models, each of which corresponds to a gene and has its own set of pairwise comparisons  \citep{yekutieli2006multiplicity, yekutieli2008false, reiner2007fdr}. A natural question is whether our techniques can be used to prove, in an asymptotic regime where the number of genes tends to infinity,  that the BH procedure is still valid when applied to the pairwise comparisons across all the genes en masse. This will certainly come down to how dependent the expression measurements between different genes are \citep{reiner2003identifying}.
One can also ask whether there are other more preferable procedures than the BH; in particular, hierarchical testing procedures have been proposed in the recent literature \citep{hFDR, hassall2018beyond}. These are beyond the scope here and we will leave them for future research.

\newpage
\begin{table}[t]
\centering
\caption{Estimates of $P(dFDP \leq \frac{\alpha}{2}\left( 1 + q_0/q\right))$ for the BH procedure in \defref{dBH}  at level $\alpha = 0.2$, based on $500$ repetitions  data are generated from the model  in \secref{numeric}, 
  for   combinations of $(m, n)$ and  $q_0/q  =  \beta = 0$.   } 
\begin{tabular}{c|cccccc|cccccc}
  \hline
$ m \backslash n$ & 20 & 40 & 100 & 200 & 400 & 600 & 20 & 40 & 100 & 200 & 400 & 600 \\ 
  \hline
& \multicolumn{6}{c|}{ \emph{effect size} =  0.05} & \multicolumn{6}{c}{ \emph{effect size} = 0.25} \\
   \hline
15 & 0.86 & 0.88 & 0.90 & 0.91 & 0.92 & 0.91 & 0.88 & 0.94 & 0.99 & 1.00 & 1.00 & 1.00 \\ 
  30 & 0.85 & 0.87 & 0.91 & 0.91 & 0.93 & 0.98 & 0.96 & 0.99 & 1.00 & 1.00 & 1.00 & 1.00 \\ 
  50 & 0.78 & 0.87 & 0.90 & 0.91 & 0.95 & 0.99 & 0.96 & 1.00 & 1.00 & 1.00 & 1.00 & 1.00 \\ 
  70 & 0.75 & 0.85 & 0.88 & 0.93 & 0.97 & 1.00 & 0.97 & 1.00 & 1.00 & 1.00 & 1.00 & 1.00 \\ 
  90 & 0.71 & 0.83 & 0.87 & 0.90 & 0.98 & 0.99 & 0.97 & 1.00 & 1.00 & 1.00 & 1.00 & 1.00 \\ 
  120 & 0.71 & 0.83 & 0.89 & 0.92 & 0.98 & 1.00 & 0.99 & 1.00 & 1.00 & 1.00 & 1.00 & 1.00 \\ 
   \hline
\hline
& \multicolumn{6}{c|}{ \emph{effect size} =  0.45} & \multicolumn{6}{c}{ \emph{effect size} = 0.65} \\
   \hline
15 & 0.97 & 1.00 & 1.00 & 1.00 & 1.00 & 1.00 & 0.99 & 1.00 & 1.00 & 1.00 & 1.00 & 1.00 \\ 
  30 & 1.00 & 1.00 & 1.00 & 1.00 & 1.00 & 1.00 & 1.00 & 1.00 & 1.00 & 1.00 & 1.00 & 1.00 \\ 
  50 & 1.00 & 1.00 & 1.00 & 1.00 & 1.00 & 1.00 & 1.00 & 1.00 & 1.00 & 1.00 & 1.00 & 1.00 \\ 
  70 & 1.00 & 1.00 & 1.00 & 1.00 & 1.00 & 1.00 & 1.00 & 1.00 & 1.00 & 1.00 & 1.00 & 1.00 \\ 
  90 & 1.00 & 1.00 & 1.00 & 1.00 & 1.00 & 1.00 & 1.00 & 1.00 & 1.00 & 1.00 & 1.00 & 1.00 \\ 
  120 & 1.00 & 1.00 & 1.00 & 1.00 & 1.00 & 1.00 & 1.00 & 1.00 & 1.00 & 1.00 & 1.00 & 1.00 \\ 
   \hline
\end{tabular}
  \label{tab:FDP_beta_0_alpha_0.2}
\end{table}
\begin{table}[ht]
\centering
\caption{Estimates of $dFDR$ for the BH procedure in \defref{dBH}  at level $\alpha = 0.2$, based on $500$ repetitions of  data  generated from the  model  in \secref{numeric}, 
  for   combinations of $(m, n)$ and  $q_0/q = \beta =  0$; 
  the dFDR target is  $\frac{\alpha}{2}\left( 1 + q_0/q\right) = 0.1$.   
  } 
\begin{tabular}{c|cccccc|cccccc}
  \hline
$ m \backslash n$ & 20 & 40 & 100 & 200 & 400 & 600 & 20 & 40 & 100 & 200 & 400 & 600 \\ 
  \hline
& \multicolumn{6}{c|}{ \emph{effect size} =  0.05} & \multicolumn{6}{c}{ \emph{effect size} = 0.25} \\
   \hline
15 & 0.08 & 0.05 & 0.04 & 0.03 & 0.03 & 0.03 & 0.04 & 0.02 & 0.02 & 0.01 & 0.01 & 0.01 \\ 
  30 & 0.07 & 0.05 & 0.03 & 0.02 & 0.03 & 0.02 & 0.03 & 0.02 & 0.01 & 0.01 & 0.01 & 0.01 \\ 
  50 & 0.09 & 0.05 & 0.03 & 0.03 & 0.02 & 0.02 & 0.03 & 0.02 & 0.01 & 0.01 & 0.01 & 0.01 \\ 
  70 & 0.11 & 0.06 & 0.03 & 0.03 & 0.03 & 0.02 & 0.03 & 0.02 & 0.01 & 0.01 & 0.01 & 0.01 \\ 
  90 & 0.12 & 0.06 & 0.03 & 0.03 & 0.02 & 0.02 & 0.03 & 0.02 & 0.02 & 0.01 & 0.01 & 0.01 \\ 
  120 & 0.11 & 0.06 & 0.03 & 0.03 & 0.03 & 0.02 & 0.03 & 0.02 & 0.02 & 0.01 & 0.01 & 0.01 \\ 
   \hline
\hline
& \multicolumn{6}{c|}{ \emph{effect size} =  0.45} & \multicolumn{6}{c}{ \emph{effect size} = 0.65} \\
   \hline
15 & 0.02 & 0.01 & 0.01 & 0.01 & 0.01 & 0.01 & 0.02 & 0.01 & 0.01 & 0.01 & 0.00 & 0.00 \\ 
  30 & 0.02 & 0.01 & 0.01 & 0.01 & 0.00 & 0.00 & 0.01 & 0.01 & 0.01 & 0.00 & 0.00 & 0.00 \\ 
  50 & 0.02 & 0.01 & 0.01 & 0.01 & 0.00 & 0.00 & 0.01 & 0.01 & 0.01 & 0.00 & 0.00 & 0.00 \\ 
  70 & 0.02 & 0.01 & 0.01 & 0.01 & 0.00 & 0.00 & 0.01 & 0.01 & 0.01 & 0.00 & 0.00 & 0.00 \\ 
  90 & 0.02 & 0.01 & 0.01 & 0.01 & 0.00 & 0.00 & 0.01 & 0.01 & 0.01 & 0.00 & 0.00 & 0.00 \\ 
  120 & 0.02 & 0.01 & 0.01 & 0.01 & 0.00 & 0.00 & 0.02 & 0.01 & 0.01 & 0.00 & 0.00 & 0.00 \\ 
   \hline
\end{tabular}
   \label{tab:FDR_beta_0_alpha_0.2}
\end{table}

\begin{table}[ht]
\centering
\caption{Estimates of $P(dFDP \leq \frac{\alpha}{2}\left( 1 + q_0/q\right))$ for the BH procedure in \defref{dBH}  at level $\alpha = 0.2$, based on $500$ repetitions  data are generated from the model  in \secref{numeric}, 
  for   combinations of $(m, n)$ and  $q_0/q  \approx \beta = 0.25$.     } 
\begin{tabular}{c|cccccc|cccccc}
  \hline
$ m \backslash n$ & 20 & 40 & 100 & 200 & 400 & 600 & 20 & 40 & 100 & 200 & 400 & 600 \\ 
  \hline
& \multicolumn{6}{c|}{ \emph{effect size} =  0.05} & \multicolumn{6}{c}{ \emph{effect size} = 0.25} \\
   \hline
15 & 0.84 & 0.85 & 0.88 & 0.89 & 0.86 & 0.86 & 0.81 & 0.83 & 0.84 & 0.90 & 0.95 & 0.95 \\ 
  30 & 0.82 & 0.84 & 0.87 & 0.85 & 0.81 & 0.82 & 0.77 & 0.83 & 0.95 & 0.98 & 0.98 & 0.99 \\ 
  50 & 0.76 & 0.85 & 0.88 & 0.84 & 0.82 & 0.82 & 0.72 & 0.81 & 0.93 & 0.97 & 1.00 & 1.00 \\ 
  70 & 0.73 & 0.83 & 0.85 & 0.86 & 0.81 & 0.79 & 0.66 & 0.80 & 0.92 & 0.97 & 0.99 & 1.00 \\ 
  90 & 0.69 & 0.81 & 0.84 & 0.85 & 0.83 & 0.82 & 0.63 & 0.78 & 0.94 & 0.99 & 1.00 & 1.00 \\ 
  120 & 0.67 & 0.81 & 0.86 & 0.86 & 0.80 & 0.85 & 0.67 & 0.85 & 0.99 & 1.00 & 1.00 & 1.00 \\ 
   \hline
\hline
& \multicolumn{6}{c|}{ \emph{effect size} =  0.45} & \multicolumn{6}{c}{ \emph{effect size} = 0.65} \\
   \hline
15 & 0.76 & 0.84 & 0.93 & 0.96 & 0.97 & 0.98 & 0.83 & 0.88 & 0.94 & 0.97 & 0.99 & 0.99 \\ 
  30 & 0.88 & 0.94 & 0.98 & 0.99 & 0.99 & 1.00 & 0.94 & 0.97 & 0.99 & 0.99 & 1.00 & 1.00 \\ 
  50 & 0.84 & 0.94 & 0.99 & 1.00 & 1.00 & 1.00 & 0.95 & 0.97 & 0.99 & 1.00 & 1.00 & 1.00 \\ 
  70 & 0.83 & 0.95 & 0.99 & 0.99 & 1.00 & 1.00 & 0.93 & 0.99 & 1.00 & 0.99 & 1.00 & 1.00 \\ 
  90 & 0.87 & 0.95 & 0.99 & 1.00 & 1.00 & 1.00 & 0.95 & 0.99 & 1.00 & 1.00 & 1.00 & 1.00 \\ 
  120 & 0.92 & 0.99 & 1.00 & 1.00 & 1.00 & 1.00 & 0.98 & 1.00 & 1.00 & 1.00 & 1.00 & 1.00 \\ 
   \hline
\end{tabular}
\label{tab:FDP_beta_0.25_alpha_0.2}
\end{table}

\begin{table}[ht]
\centering
\caption{Estimates of $dFDR$ for the BH procedure in \defref{dBH}  at level $\alpha = 0.2$, based on $500$ repetitions of  data  generated from the model  in \secref{numeric}, 
  for   combinations of $(m, n)$ and  $q_0/q \approx  \beta = 0.25$; 
  the dFDR target is  $\frac{\alpha}{2}\left( 1 + q_0/q\right) \approx 0.125$.  } 
\begin{tabular}{c|cccccc|cccccc}
  \hline
$ m \backslash n$ & 20 & 40 & 100 & 200 & 400 & 600 & 20 & 40 & 100 & 200 & 400 & 600 \\ 
  \hline
& \multicolumn{6}{c|}{ \emph{effect size} =  0.05} & \multicolumn{6}{c}{ \emph{effect size} = 0.25} \\
   \hline
15 & 0.10 & 0.08 & 0.07 & 0.05 & 0.06 & 0.05 & 0.07 & 0.05 & 0.05 & 0.05 & 0.04 & 0.04 \\ 
  30 & 0.09 & 0.08 & 0.06 & 0.05 & 0.06 & 0.06 & 0.07 & 0.06 & 0.06 & 0.05 & 0.05 & 0.05 \\ 
  50 & 0.15 & 0.08 & 0.06 & 0.07 & 0.06 & 0.06 & 0.08 & 0.07 & 0.06 & 0.06 & 0.06 & 0.05 \\ 
  70 & 0.16 & 0.10 & 0.07 & 0.05 & 0.06 & 0.06 & 0.10 & 0.08 & 0.06 & 0.06 & 0.06 & 0.06 \\ 
  90 & 0.18 & 0.11 & 0.08 & 0.06 & 0.05 & 0.06 & 0.11 & 0.08 & 0.06 & 0.06 & 0.06 & 0.06 \\ 
  120 & 0.18 & 0.12 & 0.06 & 0.06 & 0.06 & 0.06 & 0.10 & 0.08 & 0.06 & 0.06 & 0.06 & 0.06 \\ 
   \hline
\hline
& \multicolumn{6}{c|}{ \emph{effect size} =  0.45} & \multicolumn{6}{c}{ \emph{effect size} = 0.65} \\
   \hline
15 & 0.06 & 0.05 & 0.05 & 0.04 & 0.04 & 0.04 & 0.05 & 0.05 & 0.04 & 0.04 & 0.04 & 0.04 \\ 
  30 & 0.06 & 0.06 & 0.05 & 0.05 & 0.05 & 0.05 & 0.06 & 0.06 & 0.05 & 0.05 & 0.05 & 0.05 \\ 
  50 & 0.07 & 0.06 & 0.06 & 0.05 & 0.05 & 0.05 & 0.06 & 0.06 & 0.05 & 0.05 & 0.05 & 0.05 \\ 
  70 & 0.08 & 0.07 & 0.06 & 0.06 & 0.06 & 0.05 & 0.07 & 0.06 & 0.06 & 0.05 & 0.05 & 0.05 \\ 
  90 & 0.08 & 0.07 & 0.06 & 0.06 & 0.06 & 0.05 & 0.07 & 0.06 & 0.06 & 0.05 & 0.05 & 0.05 \\ 
  120 & 0.08 & 0.07 & 0.06 & 0.06 & 0.05 & 0.05 & 0.07 & 0.06 & 0.05 & 0.05 & 0.05 & 0.05 \\ 
   \hline
\end{tabular}
\label{tab:FDR_beta_0.25_alpha_0.2}
\end{table}

\begin{table}[ht]
\centering
\caption{Estimates of $P(dFDP \leq \frac{\alpha}{2}\left( 1 + q_0/q\right))$ for the BH procedure in \defref{dBH}  at level $\alpha = 0.2$, based on $500$ repetitions  data are generated from the model  in \secref{numeric}, 
  for   combinations of $(m, n)$ and  $q_0/q  \approx  \beta = 0.5$.   } 
\begin{tabular}{c|cccccc|cccccc}
  \hline
$ m \backslash n$& 20 & 40 & 100 & 200 & 400 & 600 & 20 & 40 & 100 & 200 & 400 & 600 \\ 
  \hline
& \multicolumn{6}{c|}{ \emph{effect size} =  0.05} & \multicolumn{6}{c}{ \emph{effect size} = 0.25} \\
   \hline
15 & 0.83 & 0.84 & 0.88 & 0.87 & 0.84 & 0.85 & 0.82 & 0.80 & 0.76 & 0.72 & 0.76 & 0.73 \\ 
  30 & 0.81 & 0.85 & 0.86 & 0.85 & 0.80 & 0.77 & 0.75 & 0.71 & 0.73 & 0.78 & 0.78 & 0.80 \\ 
  50 & 0.73 & 0.84 & 0.86 & 0.84 & 0.80 & 0.79 & 0.67 & 0.66 & 0.72 & 0.77 & 0.78 & 0.82 \\ 
  70 & 0.72 & 0.81 & 0.84 & 0.86 & 0.78 & 0.76 & 0.59 & 0.65 & 0.71 & 0.79 & 0.80 & 0.83 \\ 
  90 & 0.67 & 0.80 & 0.80 & 0.84 & 0.81 & 0.75 & 0.55 & 0.64 & 0.73 & 0.84 & 0.89 & 0.91 \\ 
  120 & 0.65 & 0.80 & 0.86 & 0.83 & 0.81 & 0.78 & 0.54 & 0.63 & 0.80 & 0.88 & 0.92 & 0.95 \\ 
   \hline
\hline
& \multicolumn{6}{c|}{ \emph{effect size} =  0.45} & \multicolumn{6}{c}{ \emph{effect size} = 0.65} \\
   \hline
15 & 0.74 & 0.71 & 0.74 & 0.75 & 0.79 & 0.77 & 0.71 & 0.69 & 0.75 & 0.77 & 0.80 & 0.79 \\ 
  30 & 0.68 & 0.71 & 0.80 & 0.82 & 0.80 & 0.82 & 0.71 & 0.76 & 0.81 & 0.83 & 0.83 & 0.84 \\ 
  50 & 0.62 & 0.69 & 0.80 & 0.83 & 0.84 & 0.87 & 0.67 & 0.73 & 0.83 & 0.85 & 0.86 & 0.88 \\ 
  70 & 0.60 & 0.71 & 0.78 & 0.85 & 0.86 & 0.88 & 0.66 & 0.76 & 0.80 & 0.87 & 0.90 & 0.89 \\ 
  90 & 0.59 & 0.73 & 0.84 & 0.90 & 0.93 & 0.94 & 0.70 & 0.80 & 0.87 & 0.92 & 0.94 & 0.95 \\ 
  120 & 0.63 & 0.80 & 0.91 & 0.93 & 0.96 & 0.96 & 0.75 & 0.87 & 0.93 & 0.95 & 0.96 & 0.96 \\ 
   \hline
\end{tabular}
\label{tab:FDP_beta_0.5_alpha_0.2}
\end{table}
\begin{table}[t!]
\centering
\caption{Estimates of $dFDR$ for the BH procedure in \defref{dBH}  at level $\alpha = 0.2$, based on $500$ repetitions of  data  generated from the model  in \secref{numeric}, 
  for   combinations of $(m, n)$ and  $q_0/q \approx \beta =  0.5$; 
  the dFDR target is  $\frac{\alpha}{2}\left( 1 + q_0/q\right) \approx \beta = 0.15$.  } 
\begin{tabular}{c|cccccc|cccccc}
  \hline
$ m \backslash n$ & 20 & 40 & 100 & 200 & 400 & 600 & 20 & 40 & 100 & 200 & 400 & 600 \\ 
  \hline
& \multicolumn{6}{c|}{ \emph{effect size} =  0.05} & \multicolumn{6}{c}{ \emph{effect size} = 0.25} \\
   \hline
15 & 0.13 & 0.12 & 0.08 & 0.09 & 0.10 & 0.08 & 0.10 & 0.10 & 0.08 & 0.09 & 0.08 & 0.09 \\ 
  30 & 0.14 & 0.11 & 0.09 & 0.09 & 0.09 & 0.10 & 0.11 & 0.10 & 0.10 & 0.10 & 0.10 & 0.10 \\ 
  50 & 0.19 & 0.12 & 0.10 & 0.10 & 0.09 & 0.08 & 0.15 & 0.12 & 0.11 & 0.11 & 0.10 & 0.10 \\ 
  70 & 0.21 & 0.15 & 0.11 & 0.08 & 0.10 & 0.09 & 0.17 & 0.12 & 0.11 & 0.11 & 0.11 & 0.11 \\ 
  90 & 0.24 & 0.15 & 0.13 & 0.08 & 0.08 & 0.08 & 0.16 & 0.12 & 0.11 & 0.10 & 0.10 & 0.10 \\ 
  120 & 0.24 & 0.14 & 0.08 & 0.09 & 0.07 & 0.08 & 0.17 & 0.13 & 0.11 & 0.10 & 0.10 & 0.10 \\ 
   \hline
\hline
& \multicolumn{6}{c|}{ \emph{effect size} =  0.45} & \multicolumn{6}{c}{ \emph{effect size} = 0.65} \\
   \hline
15 & 0.10 & 0.10 & 0.09 & 0.09 & 0.08 & 0.09 & 0.10 & 0.10 & 0.09 & 0.09 & 0.08 & 0.09 \\ 
  30 & 0.11 & 0.11 & 0.10 & 0.10 & 0.10 & 0.10 & 0.11 & 0.10 & 0.10 & 0.09 & 0.10 & 0.10 \\ 
  50 & 0.13 & 0.11 & 0.10 & 0.10 & 0.10 & 0.10 & 0.12 & 0.11 & 0.10 & 0.10 & 0.10 & 0.10 \\ 
  70 & 0.14 & 0.12 & 0.11 & 0.10 & 0.11 & 0.10 & 0.13 & 0.11 & 0.10 & 0.10 & 0.10 & 0.10 \\ 
  90 & 0.14 & 0.11 & 0.11 & 0.10 & 0.10 & 0.10 & 0.12 & 0.11 & 0.10 & 0.10 & 0.10 & 0.10 \\ 
  120 & 0.13 & 0.11 & 0.10 & 0.10 & 0.10 & 0.10 & 0.12 & 0.11 & 0.10 & 0.10 & 0.10 & 0.10 \\ 
   \hline
\end{tabular}
\label{tab:FDR_beta_0.5_alpha_0.2}
\end{table}

\newpage
\ \
\newpage

\appendix

\section{Maximal inequality of the sample variances}

The following inequality has been used in the proofs of \secsref{main} and \secssref{ULLNpf}.

\begin{lemma}[Maximal inequality of the sample variances] \label{lem:max_ineq_variances}
Under \assumpref{moment_regime}, we have 
\[
P\left(\max_{1 \leq i \leq m}\left|\frac{\hat{\sigma}_i^2}{\sigma_i^2}  - 1 \right| > (\log m)^{-2}\right) = O(n^{-r - c}) \text{ as } m, n \longrightarrow \infty.
\]
\end{lemma}

\begin{proof}
Note that for all $1 \leq i \leq m$,
\[
\frac{\hat{\sigma}_i^2}{\sigma_i^2} = \frac{\sum_{k = 1}^n (X_{ki}- \bar{X}_i)^2}{(n-1)\sigma_i^2} = \frac{\sum_{k = 1}^n (X_{ki}- \mu_i)^2}{(n-1)\sigma_i^2}- \frac{n(\bar{X}_i - \mu_i)^2}{(n-1)\sigma^2_i},
\]
which in turn gives
\begin{align*}
\left|\frac{\hat{\sigma}_i^2}{\sigma_i^2}  - 1 \right| &\leq  \left| \frac{n(\bar{X}_i - \mu_i)^2}{(n-1)\sigma^2_i}\right| 
+ \left|\frac{\sum_{k = 1}^n [(X_{ki}- \mu_i)^2 - \sigma_i^2]}{(n-1)\sigma_i^2} \right|
+ \left| \frac{n}{n-1}  - 1 \right|\\
& = \left|\frac{\sum_{k=1}^n (X_{ki} - \mu_i)}{\sigma_i\sqrt{n(n-1)}} \right|^2 +  \frac{1}{n-1}\left|\sum_{k=1}^n\left[\frac{ (X_{ki}- \mu_i)^2 }{\sigma_i^2} - 1\right] \right|
+ \frac{1}{n-1}.
\end{align*}
Hence, by a union bound, we get that
\begin{multline} \label{sigmaRatioBdd}
P\left(\max_{1 \leq i \leq m}\left|\frac{\hat{\sigma}_i^2}{\sigma_i^2}  - 1 \right| > (\log m)^{-2}\right)
 \leq P\left(\frac{1}{n-1} > \frac{(\log m)^{-2}}{3}\right)  \\
+ P\left( \max_{1\leq i \leq m}\left| \sum_{k = 1}^n\left[\frac{X _{ki} - \mu_i}{\sigma_i} \right]\right|>  \sqrt{\frac{n(n-1)}{3}} (\log m)^{-1}\right)  \\
+  P\left(\max_{1 \leq i \leq m}\left|\sum_{k = 1}^n\left[\frac{ (X_{ki}- \mu_i)^2 }{\sigma_i^2} - 1\right]\right|  > \frac{(\log m)^{-2}(n-1)}{3}\right)  .
\end{multline}
Since $m \leq C n^r$ by \assumpref{moment_regime},  the first term on the right hand side in \eqref{sigmaRatioBdd} is zero for large enough $n$ (and $m$), and it suffices to show the last two terms are of order $O(n^{-r-c})$. The first term
can be bounded as
\begin{multline*}
P\left( \max_{ 1 \leq i \leq m}\left| \sum_{k = 1}^n\left[\frac{X _{ki} - \mu_i}{\sigma_i} \right]\right|>  \sqrt{\frac{n(n-1)}{3}} (\log m)^{-1}\right) \\
\leq \sum_{i=1}^mP\left(\sum_{k= 1}^n\left|\frac{X _{ki} - \mu_i}{\sigma_i}\right| > \sqrt{\frac{n(n-1)}{3}} (\log m)^{-1} \right)\\
\lesssim m \frac{n^{1 + 2r + \nu/2}(\log m)^{2 + 4r + \nu}}{n^{2 + 4r + \nu}} = \frac{m (\log m)^{2 + 4r + \nu}}{n^{1 + 2r + \nu/2}} \lesssim n^{- r- c},
\end{multline*}
for some $c>0$, where the  inequalities ``$\lesssim$" follow from \citet[Theorem 3]{Rosenthal}'s inequality and  \assumpref{moment_regime}.
For the second term, we have
\begin{multline}\label{2ndterm}
P\left(\max_{1\leq i \leq m}\left|\sum_{k = 1}^n\left[\frac{ (X_{ki}- \mu_i)^2 }{\sigma_i^2} - 1\right]\right|  > \frac{(\log m)^{-2}(n-1)}{3}\right)\leq \\
\sum_{i=1}^m P\left(\left|\sum_{k = 1}^n\left[\frac{ (X_{ki}- \mu_i)^2 }{\sigma_i^2} - 1\right]\right|  > \frac{(\log m)^{-2}(n-1)}{3}\right).
\end{multline}
 For $1 + 2r + \nu/2 \leq 2$, by \citet[Theorem 2]{eseenBahr}'s inequality, $
\mathbb{E}[|\sum_{k = 1}^n( \frac{(X_{ki} -\mu_i)^2}{\sigma_i^2}- 1)|^{1+ 2r + \nu/2} ]  \lesssim n,
$
and hence with the right hand side of \eqref{2ndterm} can be bounded by a term of the order $O(m (\log m)^{4r + 2 + \nu}  n^{- 2r - \nu/2}) = O( n^{- r-c})$. 
 If $1+ 2r + \nu/2 > 2$, then we can apply a Fuk-Nagaev type inequality \citep{Liu2009}[Lemma 6.1] and  \assumpref{moment_regime} to give, for large enough $n$, 
\begin{align*}
&\sum_{i=1}^mP\left(\left|\sum_{k = 1}^n\left[\frac{ (X_{ki}- \mu_i)^2 }{\sigma_i^2} - 1\right]\right|  > \frac{(n-1)}{3 (\log m)^2}\right)
\\
&\leq \sum_{i=1}^m \left\{\sum_{k= 1}^n P\left(\left|\frac{ (X_{ki}- \mu_i)^2 }{\sigma_i^2} - 1\right| > \frac{c(n-1)}{3 (\log m)^2}  \right) +
\exp\left( - \frac{n}{C (\log m)^4}\right) + C n^{- \max(2r+ \nu, 2)} \right\}\\
&\lesssim m \left(\frac{(\log m)^{2+4r + \nu}}{n^{ 2r + \nu/2}} +     n^{- \max(2r+ \nu, 2)}       \right) = O( n ^{-r - c})
\end{align*}
for some small enough $c > 0$ and large enough $C >0 $.


\end{proof}

\section{Further proofs for \secref{ULLNpf}}

We first introduce a lemma that will be used to finish the proofs.
\begin{lemma} \label{lem:tech} Consider an asymptotic regime where $m, n \longrightarrow \infty$. For any constant $c' > \sqrt{2}$ and a sequence $b_n = o(1)$, it must be true that
\begin{multline*}
\bar{\Phi}^{-1}(\bar{\Phi}(t) (1 + b_n)) \leq c' \sqrt{\log m} \\
\text{ for all } t \in [0, \sqrt{2 \log m}] \text{ and a sufficiently large } n \; (\text{and } m).
\end{multline*}
\end{lemma}
\begin{proof}
Let $t_{b_n} = \bar{\Phi}^{-1}(\bar{\Phi}(t) (1 + b_n))$. Since $t \leq \sqrt{2 \log m}$,  it must be always true that $t_{b_n}  \leq a_m$ for a sufficiently large $n$, where $a_m = \bar{\Phi}^{-1}(\bar{\Phi}(\sqrt{2 \log m}) (1 + b_n))$. Since $s\bar{\Phi}(s)/\phi(s) \rightarrow 1$ as $s \longrightarrow \infty$, it must be that
\begin{align*}
\frac{\phi(a_m)\sqrt{2 \log m}}{a_m \phi(\sqrt{2 \log m})}&= \frac{  \bar{\Phi}(a_m) \phi(a_m)}{\bar{\Phi}(a_m)a_m} \times \frac{\sqrt{2 \log m}}{\phi(\sqrt{2 \log m})}
\\
&=  \frac{\bar{\Phi}(\sqrt{2 \log m})\sqrt{2 \log m}}{\phi(\sqrt{2 \log m})} \times \frac{\phi(a_m)}{\bar{\Phi}(a_m)a_m} \times (1 + b_n) \\
&= (1+o(1))(1 + b_n) = 1 + o(1).
\end{align*}
By taking $\log$ on both side of the preceding display, one immediately gets that
\[
\frac{a_m^2}{2} =\log \left(\sqrt{2 \log m}\right) - \log a_m+ \frac{2 \log m}{2} -  \log \Bigl(1 + o(1)\Bigr),
\]
which, in light of the fact that $\log a_m \longrightarrow \infty$ as $n, m \longrightarrow \infty$, implies that
\[
t_{b_n} \leq a_m \leq  \sqrt{2 \log m + 2 \log(\sqrt{2 \log m})  + o(1)} \leq c' \sqrt{\log m}
\]
for a large enough $n$ (and hence $m$) and $c' > \sqrt{2}$.
\end{proof}

\subsection{Proof of \eqref{switchbk}} \label{app:delicate}
We will  show that for a sufficiently large $n$, the event equivalences 
\[
\bar{p}_{ij, L} \leq \bar{\Phi}(t) \iff  - \bar{T}_{ij} \geq \bar{\Phi}^{-1} (\bar{\Phi}(t)  (1 + \epsilon_{|\bar{T}_{ij}|})
\]
and
\[
\bar{p}_{ij, U} \leq  \bar{\Phi}(t) \iff \bar{T}_{ij} \geq   \bar{\Phi}^{-1}( \bar{\Phi}(t)(1 + \epsilon_{|\bar{T}_{ij}|})),
\]
with $\epsilon_{|\bar{T}_{ij}|}$ having the property in  \eqref{eptij}.
To show this we first note the following two events are identical, by their definitions:
\begin{equation} \label{events}
\{\bar{p}_{ij, L} \leq \bar{\Phi}(t)\} = \{ 1 - {F}_{ij} (  - \bar{T}_{ij})\leq \bar{\Phi}(t)\}.
\end{equation}
Fix a constant $c' \in (\sqrt{2}, 2)$. For a sufficiently large $n$ the event in \eqref{events} can be realized for all $ 0 \leq t \leq \sqrt{2 \log m}$ if $- \bar{T}_{ij} \geq c'\sqrt{\log m}$   since, with \assumpref{ref_null},
\[
1 - {F}_{ij} ( - \bar{T}_{ij}) \leq  1 - {F}_{ij} ( c'\sqrt{\log m}) = \bar{\Phi}(c'\sqrt{\log m}) (1 + o(1)) \leq \bar{\Phi} (\sqrt{2 \log m}) \leq  \bar{\Phi} (t)
\]
when $m$ is large enough. But since the event in  \eqref{events} may also be true for some $- \bar{T}_{ij} < c'\sqrt{\log m}$, using  \assumpref{ref_null} and the identity \eqref{events} again,  we can conclude, \emph{conditioning on} $- \bar{T}_{ij} < c'\sqrt{\log m}$, the equivalence of events
\begin{multline*}
\bar{p}_{ij, L} \leq \bar{\Phi}(t) \iff  \frac{1 - F_{ij} (- \bar{T}_{ij})}{\bar{\Phi}(- \bar{T}_{ij})} \bar{\Phi}(- \bar{T}_{ij}) \leq \bar{\Phi}(t) 
 \iff - \bar{T}_{ij} \geq \bar{\Phi}^{-1} (\bar{\Phi}(t)  (1 + \epsilon_{|\bar{T}_{ij}|})).
\end{multline*}
for large enough $n$, where $\epsilon_{|\bar{T}_{ij}|}$ has the property in \eqref{eptij}. However, in light of \lemref{tech}, for large enough $n$, $- \bar{T}_{ij} \geq c'\sqrt{\log m}$ will necessarily imply $- \bar{T}_{ij} \geq \bar{\Phi}^{-1} (\bar{\Phi}(t)  (1 + \epsilon_{|\bar{T}_{ij}|}))$, and hence the train of equivalence in the preceding display is also true \emph{without} conditioning on $- \bar{T}_{ij} < c'\sqrt{\log m}$.  By a completely analogous argument we also have the equivalence of events
\[
\bar{p}_{ij, U} \leq  \bar{\Phi}(t) \iff \bar{T}_{ij} \geq   \bar{\Phi}^{-1}( \bar{\Phi}(t)(1 + \epsilon_{|\bar{T}_{ij}|}))
\]
for large enough $n$.

\subsection{Proof of \lemref{BermanBdd}} \label{app:pfBerman}
For $\epsilon_n$ is as in  \eqref{eptij},  we let 
\[
t_U = \bar{\Phi}^{-1} (\bar{\Phi} (t)( 1 - \epsilon_n))\text{ and }t_L = \bar{\Phi}^{-1} (\bar{\Phi} (t)( 1 + \epsilon_n)),
\]
where   the subscripts are suggestive of the fact that $t_L$ is always less than $t_U$. We also note that
\begin{equation} \label{sandwich}
t_L \leq
 t_{ij} \leq
t_U \text{ for all } (i, j) \text{ and } t\in[0, \sqrt{2\log m}].
\end{equation}

\begin{proof} [Proof of \lemref{BermanBdd}$(i)$] In fact, it suffices to show that for sufficiently large $m$,
\begin{equation} \label{lower}
\frac{P(\bar{T}_{ij}\geq t_U)}{\bar{\Phi}(t)}=1 +O(n^{-c})\end{equation}
\begin{equation} \label{upper}
\frac{P(\bar{T}_{ij}\geq t_L)}{\bar{\Phi}(t)}=1 +O(n^{-c})
\end{equation}
both uniformly in $t \in [0, \sqrt{2 \log m}]$ and all $(i, j)$, in light of \eqref{sandwich}.

 \eqref{lower}: By \lemref{tech}, we note that for large enough $n$,  it is true that $ t_L, t_U \leq c'\sqrt{\log m}$ for a fixed constant $c' > \sqrt{2}$. Since it must be that $0 \leq t_U$,  by \assumpref{moment_regime} and  the Cram\'er-type moderate deviation for two-sample t-statistics (\lemref{changzhou}), we have
\begin{align*}
\frac{P(\bar{T}_{ij}\geq  t_U)}{\bar{\Phi} (t)} &= ( 1 - \epsilon_n)\left(1 +  O(1)\left(\frac{1 +  t_U}{n^{1/2 - 1/(4r + 2 + \nu)}}\right)^{4r + 2 + \nu}\right)\\
&= 1 + O(n^{-c})
\end{align*}
uniformly in all $(i, j)$ and $0 \leq t \leq \sqrt{2 \log m}$, for sufficiently large $n$.

 \eqref{upper}: Since $t_L$ can be a negative number, Cram\'er-type moderate deviation results (\lemref{changzhou}) is not directly applicable. However, we can do a separation argument: For any $t \in [1, \sqrt{2 \log m}]$, for large enough $n$ it must be that $t_L \geq 0$, hence with the same argument for proving \eqref{lower}, we have
\begin{equation*}
\frac{P(\bar{T}_{ij}\geq t_L)}{\bar{\Phi}(t)}=1+O(n^{-c}) \text{ uniformly in }t \in [1, \sqrt{2 \log m}] \text{ and all } (i, j).
\end{equation*}
For any $t \in [0, 1)$, we have
\begin{align} \label{disdisdis}
\frac{P(\bar{T}_{ij}\geq t_L)}{\bar{\Phi} (t)} &= \frac{P(\bar{T}_{ij}\geq t) }{\bar{\Phi} (t)} +  \frac{ P(\bar{T}_{ij}\geq t_L) - P(\bar{T}_{ij} \geq t) }{\bar{\Phi} (t)} \notag \\
 &= 1 + O(n^{-c})+ \frac{ P(\bar{T}_{ij}\geq t_L) - P(\bar{T}_{ij} \geq t) }{\bar{\Phi} (t)}
\end{align}
uniformly in $(i, j)$ and $t \in [0, 1)$, where the last equality is based on \assumpref{moment_regime} and the Cram\'er-type moderate deviation (\lemref{changzhou}). But we also note that
\begin{multline} \label{disdisidis}
\frac{ |P(\bar{T}_{ij}\geq t_L) - P(\bar{T}_{ij} \geq t) |}{\bar{\Phi} (t)}\leq \\
\frac{ |P(\bar{T}_{ij}\geq t_L) - \bar{\Phi}(t_L) |}{\bar{\Phi} (t)} + \frac{ |P(\bar{T}_{ij}\geq t) - \bar{\Phi}(t) |}{\bar{\Phi} (t)} + \frac{ |\bar{\Phi}(t) - \bar{\Phi}(t_L)  |}{\bar{\Phi} (t)}.
\end{multline}
Given that $\bar{\Phi}(t) > C > 0$ for some $C$ on interval $[0, 1)$, on the right hand side of \eqref{disdisidis} the first two terms are of order $O(n^{-c})$ by \lemref{changzhou} (compare \eqref{2sampCramer}), and the last term is so due to a bounded first derivative of $\bar{\Phi}^{-1}$ on the interval $[\bar{\Phi}(1) , \bar{\Phi}(-1)  ]$, all uniformly in $(i, j)$ and $t \in [0, 1)$.
\end{proof}

\begin{proof} [Proof of \lemref{BermanBdd}$(ii)$]
Given \eqref{sandwich}, it suffices to show, for sufficiently large $n$,
\[
P\left(T^*_{i_1j_1} \geq \tilde{t}, T^*_{i_2 j_2} \geq \tilde{t}\right) \leq C (1+ t)^{-2}\exp (-t^2 / (1+ \delta))
\]
uniformly in $0 \leq t \leq  \sqrt{2\log m}$ and  $|\{i_1, j_1\} \cap \{i_2, j_2\}|  = 1$ for some constant  $ \delta > 0$, where $\tilde{t} :=  (1 - (\log m)^{-2})t_L$.

We first recognize that for any fixed pair $i \not= j$, $T^*_{ij}$ can be rewritten as the standardized  sum
\[
T_{ij}^* = \frac{\sum_{k = 1}^{n_i} (X_{ki} - \mu_i) - \sum_{k = 1}^{n_j} n_i n_j^{-1}(X_{kj} - \mu_j)   }{\sqrt{n_i \sigma_i^2 + n_i^2 \sigma_j^2/ n_j}} = \frac{\sum_{k = 1}^{n_i \vee n_j} \eta_{k, ij} }{\sqrt{a_{ij}}},
\]
where
\[
 \eta_{k, ij}:=
  \begin{cases}
\frac{ (X_{ki} - \mu_i) - n_i n_j^{-1} (X_{kj} - \mu_j) }{ \sigma_i}& \text{if } k \leq n_i \wedge n_j \\
   \frac{(X_{ki} - \mu_i) {\bf 1}_{\{n_i > n_j\}} -  n_i n_j^{-1}(X_{kj} - \mu_j) {\bf 1}_{\{n_i < n_j\}} }{\sigma_i}     &\text{if } k >  n_i \wedge n_j  \end{cases}, 1 \leq k \leq n_i \vee n_j,
\]
are independent random variables with mean $0$, and 
\[
a_{ij} := n_i  + \frac{n_i^2 \sigma_j^2}{n_j\sigma_i^2}.
\]

Without loss of generality, we will prove the lemma by assuming that $i_1 = i_2$, and make the identification $i_1 = i_2 = i$, $j_1= j$, $j_2 = l$ with three distinct indices $i, j, l$ from now on. For each $k = 1, \dots,  n_i \vee n_j \vee n_l$, let 
\[
{\boldsymbol \eta}_k = (\eta_{k, i j  } {\bf 1}_{(k \leq n_i \vee n_j)}/\sqrt{a_{ij}} \;\ , \;\ \eta_{k, i l  }{\bf 1}_{(k \leq n_i \vee n_l)}/\sqrt{a_{il}})^T
\] and 
\[
\Sigma = \text{Cov}\Biggl(\sum_{k = 1}^{n_i \vee n_j \vee n_l}  {\boldsymbol \eta}_k{\boldsymbol \eta}_k^T\Biggr),
\] where  the dependence of ${\boldsymbol \eta}_k$ on the particular choice of the triple $(i, j, l)$ is suppressed for brevity; this gives rise to the alternative expression
\[
(T^*_{ij}, T^*_{il})^T = \sum_{k=1}^{n_i \vee n_j \vee n_k} {\boldsymbol \eta}_k
.
\] 
Later we will use the properties that
\begin{equation} \label{mean}
\mathbb{E}[{\boldsymbol \eta}_k] = (0, 0)^T,
\end{equation}
\begin{equation}\label{diagdiag}
\text{ $\Sigma =  \begin{pmatrix}
  \Sigma_{1,1} & \Sigma_{1,2}  \\
  \Sigma_{2,1} & \Sigma_{2,2} \\
 \end{pmatrix}$ is a $2\times 2$ matrix with $1$'s on the diagonal }
 \end{equation}
 and
 \begin{equation} \label{offdiag}
(1 + c_U^2)^{-1} \leq \Sigma_{1, 2}=  \text{Cov}(T^*_{ij}, T^*_{il}) = \sqrt{\left(1 + \frac{\sigma_j^2 n_i}{\sigma_i^2n_j}\right)^{-1}} \sqrt{\left(1 + \frac{\sigma_l^2 n_i}{\sigma_i^2n_l}\right)^{-1}} \leq (1 + c_L^2)^{-1},
\end{equation}
where \eqref{offdiag} follows from \assumpref{balance_homo}.
For $1 \leq k \leq n_i \vee n_j \vee n_l$, by defining the  the truncations
\[
\hat{\boldsymbol \eta}_k = (\hat{\eta}_{k, ij}, \hat{\eta}_{k, il})^T:= {\boldsymbol \eta}_k{\bf 1}_{( \|{\boldsymbol \eta}_k\| \leq (\log m)^{-4})} - \mathbb{E}[{\boldsymbol \eta}_k{\bf 1}_{( \|{\boldsymbol \eta}_k\| \leq (\log m)^{-4})} ]
\] and
\[
\tilde{\boldsymbol \eta}_k = {\boldsymbol \eta}_k - \hat{\boldsymbol \eta}_k = {\boldsymbol \eta}_k {\bf 1}_{( \|{\boldsymbol \eta}_k\|> (\log m)^{-4})} - \mathbb{E}[{\boldsymbol \eta}_k {\bf 1}_{( \|{\boldsymbol \eta}_k\| > (\log m)^{-4})} ],
\]
whose dependence  on $(i, j, l)$ is again suppressed, we get that
\begin{multline}\label{somemline}
P\left(T^*_{i j} \geq \tilde{t}, T^*_{i l} \geq \tilde{t} \right)  \leq P\left(\Biggl\|\sum_{k = 1}^{n_i \vee n_j \vee n_l} \tilde{\boldsymbol \eta}_k \Biggr\| \geq (\log m)^{-2}\right) + \\
P\left(\sum_{k = 1}^{n_i \vee n_j} \hat{\eta}_{k, ij }  \geq \tilde{t}   - (\log m)^{-2}, \sum_{k = 1}^{n_i \vee n_l} \hat{\eta}_{k, il } \geq \tilde{t}  - (\log m)^{-2}\right) .
\end{multline}
By \assumpsref{balance_homo} and \assumpssref{moment_regime}, as well as Jensen's inequality, there exists some constant $\kappa > 0$,  such that 
\begin{equation} \label{kappa_ineq}
\max_{i\neq j\neq l}\max_{1 \leq k \leq n_i \vee n_j \vee n_l} \mathbb{E}[\|{\boldsymbol \eta}_k\|^{4 r + 2 + \nu}] \leq \kappa /  n^{2r + 1+ \nu/2}, 
\end{equation}
which also implies
\begin{align}
&\max_{i\neq j \neq l}\left\|\sum_{k = 1}^{n_i \vee n_j \vee n_l} \mathbb{E}\left[ {\boldsymbol \eta}_k {\bf 1}_{( \|{\boldsymbol \eta}_k\| > (\log m)^{-4}) }\right] \right\| \notag\\
&\leq \max_{i\neq j \neq l} \sum_{k = 1}^{n_i \vee n_j \vee n_l} \mathbb{E} [\|{\boldsymbol \eta}_k\|{\bf 1}_{( \|{\boldsymbol \eta}_k\| > (\log m)^{-4})}] \notag\\
&\leq \max_{i \neq j \neq l} \sum_{k = 1}^{n_i \vee n_j \vee n_l} (\log m)^{-4} \mathbb{E}\left[ \left\|\frac{ {\boldsymbol \eta}_k}{ (\log m)^{-4}}\right\|^{4r + 2 + \nu} \right] \notag\\
&\leq \kappa  \frac{(\log m)^{16 r + 4 + 4 \nu}}{n^{2r + \nu/2}} = o((\log m)^{-2})\label{whatever}
\end{align}
Hence for a sufficiently large  $n$, in order for $\|\sum_{k=1}^{n_i \vee n_j \vee n_l} \tilde{\boldsymbol \eta}_k\|$ to be greater than $(\log m)^{-2}$, from the definition of $\tilde{\boldsymbol \eta}_k$ it can be seen that at least one of ${\boldsymbol \eta}_k {\bf 1}_{( \|{\boldsymbol \eta}_k\| >  (\log m)^{-4})}$, $k = 1, \dots, n_i \vee n_j \vee n_l$, must be non-zero, which gives
\begin{align*}
&\max_{i \neq j\neq l} P\left(\left\|\sum_{k = 1}^{n_i \vee n_j \vee n_l} \tilde{\boldsymbol \eta}_k\right\| > (\log m)^{-2}\right)\\
&\leq \max_{i\neq j\neq l}P\left(\left\|\sum_{k = 1}^{n_i \vee n_j \vee n_l}  {\boldsymbol \eta}_k {\bf 1}_{( \|{\boldsymbol \eta}_k\|> (\log m)^{-4})}\right\| > \frac{(\log m)^{-2}}{2} -  \kappa  \frac{(\log m)^{16 r + 4 + 4 \nu}}{n^{2r + \nu/2}}\right) \\
&\leq \max_{i\neq j \neq l} n P\left(\left\|  {\boldsymbol \eta}_k \right\| > \frac{(\log m)^{-2}}{2} -  \kappa  \frac{(\log m)^{16 r + 4 + 4 \nu}}{n^{2r + \nu/2}}\right)\\
&\leq \kappa n^{-2r - \nu/2} \Bigg/\Biggl( \frac{(\log m)^{-2}}{2} -  \kappa  \frac{(\log m)^{16 r + 4 + 4 \nu}}{n^{2r + \nu/2}} \Biggr)^{4r +2 +\nu}\\
& =  \kappa n^{-2r - \nu/2}  \Bigg/\Biggl( \frac{(\log m)^{-2}}{2} - o\Bigl( (\log m)^{-2}\Bigr) \Biggr)^{4r +2 +\nu} = O(n^{-r -c}).
\end{align*}
 From \eqref{somemline} it remains to show, for some $\delta > 0$,
\begin{multline*}
\max_{i \neq j\neq l}\max_{0 \leq t \leq \sqrt{2\log m}}P\left(\sum_{k = 1}^{n_i \vee n_j} \hat{\eta}_{k, ij }  \geq \tilde{t}  - (\log m)^{-2}, \sum_{k = 1}^{n_i \vee n_l} \hat{\eta}_{k, il } \geq \tilde{t} - (\log m)^{-2}\right)\\
 \leq \frac{C}{(1 +t)^2} \exp \Biggl(- \frac{t^2}{1 + \delta}\Biggr).
\end{multline*}

%

Since $\|\hat{\boldsymbol \eta}_k\|$ are bounded by  $2(\log m)^{-4}$,  by taking  $\tau = 2  (\log m)^{-4}$ for $\tau$ in \lemref{zaitsev} \citep{zaitsev}, we have
\begin{multline*}
P\left(\sum_k \hat{\eta}_{k, ij }  \geq \tilde{t} -  (\log m)^{-2}, \sum_k \hat{\eta}_{k, il }  \geq \tilde{t}  - (\log m)^{-2}\right)\\
 \leq P\left(W_1  >\tilde{t}  - 2(\log m)^{-2}, W_2 >\tilde{t}  - 2(\log m)^{-2}\right)
 + c_1 \exp\left(- \frac{(\log m)^2 }{c_2}\right)
\end{multline*}
for some absolute constants $c_1, c_2 > 0$, where $(W_1, W_2)^T$ is a bivariate Gaussian vector with mean zero and covariance structure $\hat{\Sigma} := \mathbb{E}[\sum_{k = 1}^{n_1 \vee n_j \vee n_l}\hat{\boldsymbol \eta}_k  \hat{\boldsymbol \eta}_k^T]$.  Note that the term $c_1 \exp\left(- \frac{(\log m)^2 }{c_2}\right)$ is less than $ C (1 + t)^{-2} \exp(- t^2/(1 + \delta))$ for some $C, \delta > 0$ on the range $t \in [0, \sqrt{2\log m}]$, and we are left with uniformly bounding the probability $P\left(W_1  >\tilde{t}  - 2(\log m)^{-2}, W_2 >\tilde{t}  - 2(\log m)^{-2}\right)$ in  the preceding display.

By \eqref{mean} and \eqref{whatever}, 
\[
\max_{i \neq j \neq l}\left\|\mathbb{E}\left[\sum_{k = 1}^{n_i \vee n_j \vee n_l} {\boldsymbol \eta}_k {\bf 1}_{( \|{\boldsymbol \eta}_k\| \leq (\log m)^{-4}) }\right] \right\|= o((\log m)^{-2}),
\] hence, by letting $\|\cdot\|_\infty$ denote the matrix max norm,
\begin{align}
\max_{i \neq j \neq l}\|\Sigma - \hat{\Sigma}\|_\infty &\leq \max_{i \neq j \neq l} \left\| \mathbb{E}\left[\sum_{k = 1}^{n_i \vee n_j \vee n_l}{\boldsymbol \eta}_k{\boldsymbol \eta}_k^T{\bf 1}_{( \|{\boldsymbol \eta}_k\| > (\log m)^{-4})}\right] \right\|_\infty + o((\log m)^{-4}) \notag\\
&\leq \max_{i \neq j \neq l} \sum_{k = 1}^{n_i \vee n_j \vee n_l} \mathbb{E}[\|{\boldsymbol\eta}_k\|^2 {\bf 1}_{( \|{\boldsymbol \eta}_k\| > (\log m)^{-4})}]+ o((\log m)^{-4}) \notag\\
&=  (\log m)^{-8} \max_{i \neq j \neq l} \sum_{k = 1}^{n_i \vee n_j \vee n_l} \mathbb{E}\Bigl[\Bigl\|\frac{{\boldsymbol\eta}_k}{(\log m)^{-4}}\Bigr\|^2 {\bf 1}_{( \|{\boldsymbol \eta}_k\| > (\log m)^{-4})}\Bigr]+ o((\log m)^{-4}) \notag\\
&\leq (\log m)^{16r + 4 \nu} \frac{\kappa}{ n^{2r + \nu/2}} +  o((\log m)^{-4}) = o( (\log m)^{-2}),   \label{last2bdd}
\end{align}
where the last inequality follows from \eqref{kappa_ineq}.
By \lemref{Berman} \citep{Berman}, for $0 \leq t \leq  \sqrt{2\log m}$  and large enough $n$,
\begin{multline} \label{lastbdd}
P\left(W_1  > \tilde{t}  - 2(\log m)^{-2}, W_2 >\tilde{t}  - 2(\log m)^{-2}\right) \leq \\
 \frac{C}{\left(1 + \min\left(\frac{\tilde{t}  - 2(\log m)^{-2}}{\sqrt{\hat{\Sigma}}_{11}}, \frac{\tilde{t}   - 2(\log m)^{-2}}{\sqrt{\hat{\Sigma}}_{22}}\right)\right)^2} \exp\left(  - \frac{\min\left(\frac{\tilde{t}   - 2(\log m)^{-2}}{\sqrt{\hat{\Sigma}}_{11}}, \frac{\tilde{t}   - 2(\log m)^{-2}}{\sqrt{\hat{\Sigma}}_{22}}\right)^2}{1 + \frac{ \hat{\Sigma}_{12}}{\sqrt{\hat{\Sigma}_{11}\hat{\Sigma}_{22} }}}\right).
\end{multline}
By the bound in  \eqref{diagdiag}, \eqref{offdiag}, \eqref{last2bdd} and elementary calculations, the right hand side can further be bounded by
\[
\frac{C}{(1+t)^2} \exp(- t^2/(1 + \delta))
\]
for some other constants $C > 0$ and $1 > \delta > 0$, uniformly in $(i,j,l)$ and $0 \leq t \leq \sqrt{2 \log m}$. These elementary calculations are left to the readers, but the fact that $t_L= \bar{\Phi}^{-1} (\bar{\Phi} (t)( 1 + \epsilon_n))$ converges to $t$ uniformly in $t \in [0, \sqrt{2 \log m}]$ is helpful: It is obvious how $t_L \rightarrow t$ uniformly for $t \in [0, 2]$, since $\bar{\Phi}^{-1}$ has bounded derivative on $[ \bar{\Phi}(2), \bar{\Phi}(0)]$. On the interval $t \in [2, \sqrt{2\log m}]$, by inverse function theorem and with mean value theorem,
\[
|t_L - t| \leq |\phi(t)|^{-1}|\Phi(t)|\epsilon_n \leq \frac{\epsilon_n}{t} \rightarrow 0
\]
uniformly in $t$ since $\epsilon_n = O(n^{-c})$.
 \end{proof}

\newpage
\section{Technical lemmas}

\begin{lemma}\label{lem:changzhou}
\citep[Theorem 2.3]{chang2016cramer} Let $\{Y_1, \dots, Y_{n_y}\}$ and  $\{Z_1, \dots, Z_{n_z}\}$ be two independent i.i.d  samples such that $\bE[Y_1] = \bE[Z_1] = 0$ and $\bE[|Y_1|^{2 + \delta}], \bE[|Z_1|^{2 + \delta}]   < \infty$ for some positive number $\delta \in (0, 1]$, and let $\bE[Y_1^2] = \sigma^2_y$ and $\bE[Z_1^2] = \sigma^2_z$. Consider the two sample t-statistic
\[
T = \frac{\bar{Y} - \bar{Z} }{ \sqrt{\hat{\sigma}_y^2/n_y + \hat{\sigma}_z^2/n_z}},
\]
where $\bar{Y}= n_y^{-1} \sum_{1 \leq k \leq n_y} Y_k$, $\bar{Z}= n_z^{-1} \sum_{1 \leq k \leq n_z} Z_k$,
\[
\hat{\sigma}_y^2 = \frac{\sum_{1 \leq k \leq n_y}( Y_k- \bar{Y})^2}{n_y - 1}, \; \hat{\sigma}_z^2 = \frac{\sum_{1 \leq k \leq n_z}( Z_k- \bar{Z})^2}{n_z - 1}.
\]
There exist absolute constants $A, a >0$ such that 
\[
\frac{P(T \geq s)}{\bar{\Phi}(s)} = 1 + O(1) (1 + s)^{2 + \delta} \Biggl( \frac{\bE[|Y_1|^{2 + \delta}]}{\sigma_y^{2 + \delta} n_y^{\delta/2}}+  \frac{\bE[|Z_1|^{2 + \delta}]}{\sigma_z^{2 + \delta} n_z^{\delta/2}}\Biggr),
\]
holds   for $0 \leq s \leq a \min \Bigl(\frac{\sigma_y n_y^{1/2- 1/(2+ \delta)}}{(\bE[|Y_1|^{2 + \delta}])^{1/(2+ \delta)}}, \frac{\sigma_z n_z^{1/2- 1/(2+ \delta)}}{(\bE[|Z_1|^{2 + \delta}])^{1/(2+ \delta)}}\Bigr)$, where $|O(1)| \leq A$.

\end{lemma}

\begin{lemma}\label{lem:Berman}\citep[Lemma 2]{Berman}
If $(Y, Z)$ have a bivariate  normal distribution with $\bE[Y] = \bE[Z] = 0$, $\bE[Y^2] = \bE[Z^2] = 1$ and $\bE[YZ] = r$, then
\[
\lim_{c \rightarrow \infty} \frac{P(Y > c, Z > c)}{ [2 \pi (1 - r)^{1/2} c^2]^{-1} \exp( - \frac{c^2}{1 + r})(1 + r)^{3/2}} = 1
\]
uniformly for all $r$ such that $|r| \leq \delta$, for any $\delta \in (0, 1)$.

\end{lemma}

\begin{lemma}\label{lem:zaitsev}\citep[Theorem 1.1]{zaitsev}
Let $\tau >0$,and $\xi_1, \dots, \xi_n \in \mathbb{R}^p$ be independent mean-zero random vectors such that for all $t, u \in \mathbb{R}^p$,
\[
\bE[(\xi_i^T t)^2 (\xi_i^T u)^{\omega-2} ] \leq \frac{\omega!}{2} \tau^{\omega-2} \|u\|^{\omega-2} \bE[(\xi_i^Tt)^2] \text{ for every } \omega =3,4, \dots, 
\] 
and define $S = \sum_{i=1}^n \xi_i$. Let $B$ be a Borel set in $\mathbb{R}^p$ and, for $\lambda >0$, $B^\lambda$ be its $\lambda$-neighborhood defined by
\[
B^\lambda = \Bigl\{y \in \mathbb{R}^p: \inf_{x \in B}\|x - y\|< \lambda\Bigr\}.
\]
If $\mu_{S}$ and $\mu_{\mathcal{N}_{0, \text{Cov}(S)}}$ are respectively the probability measures of $S$ and  of a $p$-variate normal distribution with mean $0$ and the same covariance structure as $S$, then for all $\lambda \geq 0$,
\begin{multline*}
\sup_B \Biggl\{\Bigl(\mu_{S}(B) - \mu_{\mathcal{N}_{0, \text{Cov}(S)}}(B^\lambda)\Bigr) \vee \Bigl( \mu_{\mathcal{N}_{0, \text{Cov}(S)}}(B)- \mu_{S}(B^\lambda) \Bigr) \Biggr\}
\leq c_1(p) \exp\left( - \frac{\lambda}{c_2(p) \tau}\right),
\end{multline*}
where $c_1(p), c_2(p) >0$ are constants depending on $p$ only, and the supremum is taken over all Borel sets.

\end{lemma}

\section{Additional numerical studies} \label{app:additional_numerical}

We provide extra simulation results for $\alpha = 0.3$, $0.1$, for the setups in \secref{numeric}. 
\begin{table}[t]
\centering
\caption{Estimates of $P(dFDP \leq \frac{\alpha}{2}\left( 1 + q_0/q\right))$ for the BH procedure in \defref{dBH}  at level $\alpha = 0.3$, based on $500$ repetitions  data are generated from the one-way ANOVA model  in \secref{numeric}, 
  for   combinations of $(m, n)$ and  $q_0/q  =  0$.   } 
\begin{tabular}{c|cccccc|cccccc}
  \hline
$ m \backslash n$ & 20 & 40 & 100 & 200 & 400 & 600 & 20 & 40 & 100 & 200 & 400 & 600 \\ 
  \hline
& \multicolumn{6}{c|}{ \emph{effect size} =  0.05} & \multicolumn{6}{c}{ \emph{effect size} = 0.25} \\
   \hline
15 & 0.79 & 0.83 & 0.87 & 0.88 & 0.91 & 0.93 & 0.89 & 0.97 & 1.00 & 1.00 & 1.00 & 1.00 \\ 
  30 & 0.79 & 0.82 & 0.90 & 0.92 & 0.96 & 0.99 & 0.99 & 1.00 & 1.00 & 1.00 & 1.00 & 1.00 \\ 
  50 & 0.73 & 0.82 & 0.89 & 0.95 & 0.98 & 1.00 & 0.99 & 1.00 & 1.00 & 1.00 & 1.00 & 1.00 \\ 
  70 & 0.69 & 0.82 & 0.87 & 0.93 & 0.99 & 1.00 & 0.99 & 1.00 & 1.00 & 1.00 & 1.00 & 1.00 \\ 
  90 & 0.67 & 0.79 & 0.85 & 0.94 & 0.99 & 1.00 & 0.99 & 1.00 & 1.00 & 1.00 & 1.00 & 1.00 \\ 
  120 & 0.64 & 0.77 & 0.88 & 0.95 & 1.00 & 1.00 & 1.00 & 1.00 & 1.00 & 1.00 & 1.00 & 1.00 \\ 
   \hline
\hline
& \multicolumn{6}{c|}{ \emph{effect size} =  0.45} & \multicolumn{6}{c}{ \emph{effect size} = 0.65} \\
   \hline
15 & 0.99 & 1.00 & 1.00 & 1.00 & 1.00 & 1.00 & 1.00 & 1.00 & 1.00 & 1.00 & 1.00 & 1.00 \\ 
  30 & 1.00 & 1.00 & 1.00 & 1.00 & 1.00 & 1.00 & 1.00 & 1.00 & 1.00 & 1.00 & 1.00 & 1.00 \\ 
  50 & 1.00 & 1.00 & 1.00 & 1.00 & 1.00 & 1.00 & 1.00 & 1.00 & 1.00 & 1.00 & 1.00 & 1.00 \\ 
  70 & 1.00 & 1.00 & 1.00 & 1.00 & 1.00 & 1.00 & 1.00 & 1.00 & 1.00 & 1.00 & 1.00 & 1.00 \\ 
  90 & 1.00 & 1.00 & 1.00 & 1.00 & 1.00 & 1.00 & 1.00 & 1.00 & 1.00 & 1.00 & 1.00 & 1.00 \\ 
  120 & 1.00 & 1.00 & 1.00 & 1.00 & 1.00 & 1.00 & 1.00 & 1.00 & 1.00 & 1.00 & 1.00 & 1.00 \\ 
   \hline
\end{tabular}
  \label{tab:FDPtab}
\end{table}
\begin{table}[ht]
\centering
\caption{Estimates of $dFDR$ for the BH procedure in \defref{dBH}  at level $\alpha = 0.3$, based on $500$ repetitions of  data  generated from the one-way ANOVA model  in \secref{numeric}, 
  for   combinations of $(m, n)$ and  $q_0/q = 0$; 
  the dFDR target is  $\frac{\alpha}{2}\left( 1 + q_0/q\right) = 0.15$.   
  } 
\begin{tabular}{c|cccccc|cccccc}
  \hline
$ m \backslash n$& 20 & 40 & 100 & 200 & 400 & 600 & 20 & 40 & 100 & 200 & 400 & 600 \\ 
  \hline
& \multicolumn{6}{c|}{ \emph{effect size} =  0.05} & \multicolumn{6}{c}{ \emph{effect size} = 0.25} \\
   \hline
15 & 0.11 & 0.08 & 0.05 & 0.05 & 0.04 & 0.04 & 0.05 & 0.04 & 0.03 & 0.02 & 0.01 & 0.01 \\ 
  30 & 0.09 & 0.07 & 0.04 & 0.04 & 0.04 & 0.04 & 0.04 & 0.03 & 0.02 & 0.02 & 0.01 & 0.01 \\ 
  50 & 0.12 & 0.07 & 0.05 & 0.04 & 0.04 & 0.04 & 0.05 & 0.04 & 0.02 & 0.02 & 0.01 & 0.01 \\ 
  70 & 0.14 & 0.07 & 0.05 & 0.04 & 0.04 & 0.04 & 0.05 & 0.04 & 0.02 & 0.02 & 0.01 & 0.01 \\ 
  90 & 0.14 & 0.08 & 0.06 & 0.04 & 0.04 & 0.04 & 0.05 & 0.04 & 0.03 & 0.02 & 0.01 & 0.01 \\ 
  120 & 0.14 & 0.08 & 0.05 & 0.04 & 0.04 & 0.04 & 0.05 & 0.04 & 0.03 & 0.02 & 0.01 & 0.01 \\ 
   \hline
\hline
& \multicolumn{6}{c|}{ \emph{effect size} =  0.45} & \multicolumn{6}{c}{ \emph{effect size} = 0.65} \\
   \hline
15 & 0.03 & 0.02 & 0.02 & 0.01 & 0.01 & 0.01 & 0.03 & 0.02 & 0.01 & 0.01 & 0.01 & 0.01 \\ 
  30 & 0.03 & 0.02 & 0.01 & 0.01 & 0.01 & 0.01 & 0.02 & 0.01 & 0.01 & 0.01 & 0.00 & 0.00 \\ 
  50 & 0.03 & 0.02 & 0.01 & 0.01 & 0.01 & 0.01 & 0.02 & 0.02 & 0.01 & 0.01 & 0.01 & 0.00 \\ 
  70 & 0.03 & 0.02 & 0.01 & 0.01 & 0.01 & 0.01 & 0.02 & 0.02 & 0.01 & 0.01 & 0.01 & 0.00 \\ 
  90 & 0.03 & 0.02 & 0.02 & 0.01 & 0.01 & 0.01 & 0.02 & 0.02 & 0.01 & 0.01 & 0.01 & 0.00 \\ 
  120 & 0.03 & 0.02 & 0.02 & 0.01 & 0.01 & 0.01 & 0.02 & 0.02 & 0.01 & 0.01 & 0.01 & 0.00 \\ 
   \hline
\end{tabular}   \label{tab:FDRtab}
\end{table}

\begin{table}[ht]
\centering
\caption{Estimates of $P(dFDP \leq \frac{\alpha}{2}\left( 1 + q_0/q\right))$ for the BH procedure in \defref{dBH}  at level $\alpha = 0.3$, based on $500$ repetitions  data are generated from the one-way ANOVA model  in \secref{numeric}, 
  for   combinations of $(m, n)$ and  $q_0/q  \approx 0.25$.     } 
\begin{tabular}{c|cccccc|cccccc}
  \hline
$ m \backslash n$ & 20 & 40 & 100 & 200 & 400 & 600 & 20 & 40 & 100 & 200 & 400 & 600 \\ 
  \hline
& \multicolumn{6}{c|}{ \emph{effect size} =  0.05} & \multicolumn{6}{c}{ \emph{effect size} = 0.25} \\
   \hline
15 & 0.77 & 0.82 & 0.83 & 0.85 & 0.81 & 0.81 & 0.76 & 0.81 & 0.90 & 0.94 & 0.98 & 0.99 \\ 
  30 & 0.77 & 0.79 & 0.83 & 0.81 & 0.79 & 0.83 & 0.80 & 0.88 & 0.98 & 1.00 & 1.00 & 1.00 \\ 
  50 & 0.68 & 0.77 & 0.80 & 0.81 & 0.79 & 0.82 & 0.74 & 0.86 & 0.98 & 1.00 & 1.00 & 1.00 \\ 
  70 & 0.66 & 0.76 & 0.80 & 0.82 & 0.79 & 0.81 & 0.70 & 0.85 & 0.98 & 0.99 & 1.00 & 1.00 \\ 
  90 & 0.62 & 0.75 & 0.78 & 0.79 & 0.80 & 0.84 & 0.70 & 0.85 & 0.98 & 0.99 & 1.00 & 1.00 \\ 
  120 & 0.59 & 0.73 & 0.79 & 0.80 & 0.80 & 0.89 & 0.77 & 0.93 & 1.00 & 1.00 & 1.00 & 1.00 \\ 
   \hline
\hline
& \multicolumn{6}{c|}{ \emph{effect size} =  0.45} & \multicolumn{6}{c}{ \emph{effect size} = 0.65} \\
   \hline
15 & 0.79 & 0.88 & 0.97 & 0.99 & 1.00 & 1.00 & 0.90 & 0.95 & 0.99 & 1.00 & 1.00 & 1.00 \\ 
  30 & 0.95 & 0.99 & 1.00 & 1.00 & 1.00 & 1.00 & 0.99 & 1.00 & 1.00 & 1.00 & 1.00 & 1.00 \\ 
  50 & 0.93 & 0.98 & 1.00 & 1.00 & 1.00 & 1.00 & 0.99 & 1.00 & 1.00 & 1.00 & 1.00 & 1.00 \\ 
  70 & 0.91 & 0.99 & 1.00 & 1.00 & 1.00 & 1.00 & 0.98 & 1.00 & 1.00 & 1.00 & 1.00 & 1.00 \\ 
  90 & 0.93 & 0.99 & 1.00 & 1.00 & 1.00 & 1.00 & 0.99 & 1.00 & 1.00 & 1.00 & 1.00 & 1.00 \\ 
  120 & 0.97 & 1.00 & 1.00 & 1.00 & 1.00 & 1.00 & 0.99 & 1.00 & 1.00 & 1.00 & 1.00 & 1.00 \\ 
   \hline
\end{tabular}
\end{table}

\begin{table}[ht]
\centering
\caption{Estimates of $dFDR$ for the BH procedure in \defref{dBH}  at level $\alpha = 0.3$, based on $500$ repetitions of  data  generated from the one-way ANOVA model  in \secref{numeric}, 
  for   combinations of $(m, n)$ and  $q_0/q \approx 0.25$; 
  the dFDR target is  $\frac{\alpha}{2}\left( 1 + q_0/q\right) \approx 0.1875$.  } 
\begin{tabular}{c|cccccc|cccccc}
  \hline
$ m \backslash n$ & 20 & 40 & 100 & 200 & 400 & 600 & 20 & 40 & 100 & 200 & 400 & 600 \\ 
  \hline
& \multicolumn{6}{c|}{ \emph{effect size} =  0.05} & \multicolumn{6}{c}{ \emph{effect size} = 0.25} \\ 
   \hline
15 & 0.15 & 0.11 & 0.09 & 0.08 & 0.08 & 0.08 & 0.11 & 0.08 & 0.07 & 0.07 & 0.07 & 0.07 \\ 
  30 & 0.13 & 0.11 & 0.08 & 0.08 & 0.09 & 0.09 & 0.10 & 0.09 & 0.09 & 0.08 & 0.08 & 0.08 \\ 
  50 & 0.19 & 0.13 & 0.10 & 0.09 & 0.10 & 0.09 & 0.12 & 0.10 & 0.09 & 0.09 & 0.08 & 0.08 \\ 
  70 & 0.20 & 0.13 & 0.10 & 0.09 & 0.10 & 0.10 & 0.13 & 0.12 & 0.10 & 0.09 & 0.09 & 0.09 \\ 
  90 & 0.22 & 0.15 & 0.11 & 0.09 & 0.09 & 0.09 & 0.13 & 0.11 & 0.10 & 0.09 & 0.09 & 0.09 \\ 
  120 & 0.23 & 0.15 & 0.10 & 0.08 & 0.10 & 0.10 & 0.14 & 0.11 & 0.10 & 0.09 & 0.09 & 0.08 \\ 
   \hline
\hline
& \multicolumn{6}{c|}{ \emph{effect size} =  0.45} & \multicolumn{6}{c}{ \emph{effect size} = 0.65} \\
   \hline
15 & 0.09 & 0.08 & 0.07 & 0.07 & 0.06 & 0.06 & 0.08 & 0.07 & 0.06 & 0.06 & 0.06 & 0.06 \\ 
  30 & 0.09 & 0.09 & 0.08 & 0.07 & 0.08 & 0.07 & 0.09 & 0.08 & 0.08 & 0.07 & 0.07 & 0.07 \\ 
  50 & 0.10 & 0.09 & 0.08 & 0.08 & 0.08 & 0.08 & 0.10 & 0.09 & 0.08 & 0.08 & 0.08 & 0.08 \\ 
  70 & 0.11 & 0.10 & 0.09 & 0.09 & 0.08 & 0.08 & 0.11 & 0.10 & 0.08 & 0.08 & 0.08 & 0.08 \\ 
  90 & 0.12 & 0.10 & 0.09 & 0.09 & 0.08 & 0.08 & 0.10 & 0.09 & 0.08 & 0.08 & 0.08 & 0.08 \\ 
  120 & 0.11 & 0.10 & 0.09 & 0.08 & 0.08 & 0.08 & 0.10 & 0.09 & 0.08 & 0.08 & 0.08 & 0.08 \\ 
   \hline
\end{tabular}
\end{table}

\begin{table}[ht]
\centering
\caption{Estimates of $P(dFDP \leq \frac{\alpha}{2}\left( 1 + q_0/q\right))$ for the BH procedure in \defref{dBH}  at level $\alpha = 0.3$, based on $500$ repetitions  data are generated from the one-way ANOVA model  in \secref{numeric}, 
  for   combinations of $(m, n)$ and  $q_0/q  \approx  0.5$.   } 
\begin{tabular}{c|cccccc|cccccc}
  \hline
$ m \backslash n$ & 20 & 40 & 100 & 200 & 400 & 600 & 20 & 40 & 100 & 200 & 400 & 600 \\ 
  \hline
& \multicolumn{6}{c|}{ \emph{effect size} =  0.05} & \multicolumn{6}{c}{ \emph{effect size} = 0.25} \\
   \hline
15 & 0.74 & 0.78 & 0.82 & 0.79 & 0.78 & 0.80 & 0.72 & 0.74 & 0.73 & 0.73 & 0.79 & 0.77 \\ 
  30 & 0.75 & 0.78 & 0.82 & 0.81 & 0.73 & 0.73 & 0.71 & 0.69 & 0.78 & 0.83 & 0.86 & 0.85 \\ 
  50 & 0.67 & 0.75 & 0.80 & 0.77 & 0.73 & 0.74 & 0.62 & 0.67 & 0.76 & 0.82 & 0.86 & 0.88 \\ 
  70 & 0.64 & 0.75 & 0.80 & 0.81 & 0.74 & 0.72 & 0.55 & 0.64 & 0.74 & 0.83 & 0.87 & 0.89 \\ 
  90 & 0.59 & 0.74 & 0.74 & 0.77 & 0.77 & 0.72 & 0.53 & 0.67 & 0.80 & 0.89 & 0.95 & 0.95 \\ 
  120 & 0.58 & 0.73 & 0.78 & 0.80 & 0.74 & 0.75 & 0.52 & 0.67 & 0.87 & 0.92 & 0.97 & 0.97 \\ 
   \hline
\hline
& \multicolumn{6}{c|}{ \emph{effect size} =  0.45} & \multicolumn{6}{c}{ \emph{effect size} = 0.65} \\
   \hline
15 & 0.67 & 0.69 & 0.75 & 0.78 & 0.85 & 0.84 & 0.68 & 0.71 & 0.78 & 0.83 & 0.86 & 0.87 \\ 
  30 & 0.70 & 0.77 & 0.85 & 0.89 & 0.90 & 0.88 & 0.76 & 0.81 & 0.88 & 0.91 & 0.90 & 0.89 \\ 
  50 & 0.65 & 0.73 & 0.85 & 0.87 & 0.92 & 0.91 & 0.76 & 0.80 & 0.90 & 0.90 & 0.93 & 0.94 \\ 
  70 & 0.64 & 0.76 & 0.83 & 0.90 & 0.92 & 0.93 & 0.69 & 0.83 & 0.87 & 0.93 & 0.94 & 0.94 \\ 
  90 & 0.67 & 0.80 & 0.91 & 0.96 & 0.97 & 0.98 & 0.78 & 0.88 & 0.95 & 0.97 & 0.98 & 0.99 \\ 
  120 & 0.72 & 0.88 & 0.97 & 0.98 & 0.99 & 0.99 & 0.84 & 0.94 & 0.98 & 0.99 & 1.00 & 0.99 \\ 
   \hline
\end{tabular}
\end{table}
\begin{table}[t!]
\centering
\caption{Estimates of $dFDR$ for the BH procedure in \defref{dBH}  at level $\alpha = 0.3$, based on $500$ repetitions of  data  generated from the one-way ANOVA model  in \secref{numeric}, 
  for   combinations of $(m, n)$ and  $q_0/q \approx 0.5$; 
  the dFDR target is  $\frac{\alpha}{2}\left( 1 + q_0/q\right) \approx 0.225$.  } 
\begin{tabular}{c|cccccc|cccccc}
  \hline
$ m \backslash n$ & 20 & 40 & 100 & 200 & 400 & 600 & 20 & 40 & 100 & 200 & 400 & 600 \\ 
  \hline
& \multicolumn{6}{c|}{ \emph{effect size} =  0.05} & \multicolumn{6}{c}{ \emph{effect size} = 0.25} \\
   \hline
15 & 0.19 & 0.16 & 0.13 & 0.14 & 0.13 & 0.12 & 0.16 & 0.13 & 0.13 & 0.13 & 0.13 & 0.13 \\ 
  30 & 0.18 & 0.16 & 0.12 & 0.12 & 0.15 & 0.14 & 0.15 & 0.15 & 0.15 & 0.15 & 0.15 & 0.15 \\ 
  50 & 0.24 & 0.18 & 0.14 & 0.14 & 0.14 & 0.13 & 0.19 & 0.17 & 0.16 & 0.16 & 0.15 & 0.15 \\ 
  70 & 0.27 & 0.18 & 0.14 & 0.11 & 0.13 & 0.14 & 0.22 & 0.18 & 0.16 & 0.16 & 0.16 & 0.16 \\ 
  90 & 0.29 & 0.19 & 0.17 & 0.12 & 0.12 & 0.13 & 0.21 & 0.18 & 0.16 & 0.16 & 0.15 & 0.15 \\ 
  120 & 0.29 & 0.19 & 0.14 & 0.11 & 0.13 & 0.13 & 0.21 & 0.18 & 0.16 & 0.16 & 0.15 & 0.15 \\ 
   \hline
\hline
& \multicolumn{6}{c|}{ \emph{effect size} =  0.45} & \multicolumn{6}{c}{ \emph{effect size} = 0.65} \\
   \hline
15 & 0.15 & 0.14 & 0.13 & 0.13 & 0.13 & 0.13 & 0.15 & 0.14 & 0.13 & 0.13 & 0.12 & 0.12 \\ 
  30 & 0.16 & 0.16 & 0.15 & 0.14 & 0.15 & 0.14 & 0.16 & 0.15 & 0.15 & 0.14 & 0.15 & 0.14 \\ 
  50 & 0.18 & 0.17 & 0.15 & 0.15 & 0.15 & 0.15 & 0.17 & 0.16 & 0.15 & 0.15 & 0.15 & 0.15 \\ 
  70 & 0.19 & 0.17 & 0.16 & 0.16 & 0.16 & 0.15 & 0.18 & 0.17 & 0.15 & 0.15 & 0.16 & 0.15 \\ 
  90 & 0.19 & 0.17 & 0.16 & 0.15 & 0.15 & 0.15 & 0.18 & 0.16 & 0.15 & 0.15 & 0.15 & 0.15 \\ 
  120 & 0.19 & 0.16 & 0.15 & 0.15 & 0.15 & 0.15 & 0.17 & 0.16 & 0.15 & 0.15 & 0.15 & 0.15 \\ 
   \hline
\end{tabular}
\end{table}

\begin{table}[t]
\centering
\caption{Estimates of $P(dFDP \leq \frac{\alpha}{2}\left( 1 + q_0/q\right))$ for the BH procedure in \defref{dBH}  at level $\alpha = 0.1$, based on $500$ repetitions  data are generated from the one-way ANOVA model  in \secref{numeric}, 
  for   combinations of $(m, n)$ and  $q_0/q  =  0$.   } 
\begin{tabular}{c|cccccc|cccccc}
  \hline
$ m \backslash n$ & 20 & 40 & 100 & 200 & 400 & 600 & 20 & 40 & 100 & 200 & 400 & 600 \\ 
  \hline
& \multicolumn{6}{c|}{ \emph{effect size} =  0.05} & \multicolumn{6}{c}{ \emph{effect size} = 0.25} \\
   \hline
15 & 0.92 & 0.94 & 0.96 & 0.94 & 0.92 & 0.92 & 0.91 & 0.90 & 0.96 & 0.99 & 1.00 & 1.00 \\ 
  30 & 0.91 & 0.93 & 0.95 & 0.94 & 0.92 & 0.95 & 0.91 & 0.98 & 1.00 & 1.00 & 1.00 & 1.00 \\ 
  50 & 0.87 & 0.94 & 0.93 & 0.92 & 0.94 & 0.95 & 0.88 & 0.98 & 1.00 & 1.00 & 1.00 & 1.00 \\ 
  70 & 0.83 & 0.92 & 0.93 & 0.94 & 0.94 & 0.98 & 0.90 & 0.99 & 1.00 & 1.00 & 1.00 & 1.00 \\ 
  90 & 0.83 & 0.91 & 0.92 & 0.91 & 0.96 & 0.97 & 0.89 & 1.00 & 1.00 & 1.00 & 1.00 & 1.00 \\ 
  120 & 0.79 & 0.90 & 0.92 & 0.91 & 0.95 & 0.98 & 0.93 & 1.00 & 1.00 & 1.00 & 1.00 & 1.00 \\ 
   \hline
\hline
& \multicolumn{6}{c|}{ \emph{effect size} =  0.45} & \multicolumn{6}{c}{ \emph{effect size} = 0.65} \\
   \hline
15 & 0.92 & 0.98 & 1.00 & 1.00 & 1.00 & 1.00 & 0.98 & 1.00 & 1.00 & 1.00 & 1.00 & 1.00 \\ 
  30 & 1.00 & 1.00 & 1.00 & 1.00 & 1.00 & 1.00 & 1.00 & 1.00 & 1.00 & 1.00 & 1.00 & 1.00 \\ 
  50 & 1.00 & 1.00 & 1.00 & 1.00 & 1.00 & 1.00 & 1.00 & 1.00 & 1.00 & 1.00 & 1.00 & 1.00 \\ 
  70 & 1.00 & 1.00 & 1.00 & 1.00 & 1.00 & 1.00 & 1.00 & 1.00 & 1.00 & 1.00 & 1.00 & 1.00 \\ 
  90 & 1.00 & 1.00 & 1.00 & 1.00 & 1.00 & 1.00 & 1.00 & 1.00 & 1.00 & 1.00 & 1.00 & 1.00 \\ 
  120 & 1.00 & 1.00 & 1.00 & 1.00 & 1.00 & 1.00 & 1.00 & 1.00 & 1.00 & 1.00 & 1.00 & 1.00 \\ 
   \hline
\end{tabular}
  \label{tab:FDPtab}
\end{table}
\begin{table}[ht]
\centering
\caption{Estimates of $dFDR$ for the BH procedure in \defref{dBH}  at level $\alpha = 0.1$, based on $500$ repetitions of  data  generated from the one-way ANOVA model  in \secref{numeric}, 
  for   combinations of $(m, n)$ and  $q_0/q = 0$; 
  the dFDR target is  $\frac{\alpha}{2}\left( 1 + q_0/q\right) = 0.05$.   
  } 
\begin{tabular}{c|cccccc|cccccc}
  \hline
$ m \backslash n$& 20 & 40 & 100 & 200 & 400 & 600 & 20 & 40 & 100 & 200 & 400 & 600 \\ 
  \hline
& \multicolumn{6}{c|}{ \emph{effect size} =  0.05} & \multicolumn{6}{c}{ \emph{effect size} = 0.25} \\
   \hline
15 & 0.04 & 0.03 & 0.02 & 0.01 & 0.02 & 0.01 & 0.02 & 0.01 & 0.01 & 0.01 & 0.00 & 0.00 \\ 
  30 & 0.04 & 0.03 & 0.01 & 0.01 & 0.01 & 0.01 & 0.01 & 0.01 & 0.01 & 0.00 & 0.00 & 0.00 \\ 
  50 & 0.06 & 0.02 & 0.01 & 0.01 & 0.01 & 0.01 & 0.02 & 0.01 & 0.01 & 0.00 & 0.00 & 0.00 \\ 
  70 & 0.08 & 0.03 & 0.02 & 0.01 & 0.01 & 0.01 & 0.02 & 0.01 & 0.01 & 0.00 & 0.00 & 0.00 \\ 
  90 & 0.07 & 0.03 & 0.02 & 0.01 & 0.01 & 0.01 & 0.02 & 0.01 & 0.01 & 0.01 & 0.00 & 0.00 \\ 
  120 & 0.08 & 0.04 & 0.02 & 0.01 & 0.01 & 0.01 & 0.02 & 0.01 & 0.01 & 0.01 & 0.00 & 0.00 \\ 
   \hline
\hline
& \multicolumn{6}{c|}{ \emph{effect size} =  0.45} & \multicolumn{6}{c}{ \emph{effect size} = 0.65} \\
   \hline
15 & 0.01 & 0.01 & 0.00 & 0.00 & 0.00 & 0.00 & 0.01 & 0.00 & 0.00 & 0.00 & 0.00 & 0.00 \\ 
  30 & 0.01 & 0.01 & 0.00 & 0.00 & 0.00 & 0.00 & 0.01 & 0.00 & 0.00 & 0.00 & 0.00 & 0.00 \\ 
  50 & 0.01 & 0.01 & 0.00 & 0.00 & 0.00 & 0.00 & 0.01 & 0.00 & 0.00 & 0.00 & 0.00 & 0.00 \\ 
  70 & 0.01 & 0.01 & 0.00 & 0.00 & 0.00 & 0.00 & 0.01 & 0.00 & 0.00 & 0.00 & 0.00 & 0.00 \\ 
  90 & 0.01 & 0.01 & 0.00 & 0.00 & 0.00 & 0.00 & 0.01 & 0.00 & 0.00 & 0.00 & 0.00 & 0.00 \\ 
  120 & 0.01 & 0.01 & 0.00 & 0.00 & 0.00 & 0.00 & 0.01 & 0.00 & 0.00 & 0.00 & 0.00 & 0.00 \\ 
   \hline
\end{tabular}
   \label{tab:FDRtab}
\end{table}

\begin{table}[ht]
\centering
\caption{Estimates of $P(dFDP \leq \frac{\alpha}{2}\left( 1 + q_0/q\right))$ for the BH procedure in \defref{dBH}  at level $\alpha = 0.1$, based on $500$ repetitions  data are generated from the one-way ANOVA model  in \secref{numeric}, 
  for   combinations of $(m, n)$ and  $q_0/q  \approx 0.25$.     } 
\begin{tabular}{c|cccccc|cccccc}
  \hline
$ m \backslash n$ & 20 & 40 & 100 & 200 & 400 & 600 & 20 & 40 & 100 & 200 & 400 & 600 \\ 
  \hline
& \multicolumn{6}{c|}{ \emph{effect size} =  0.05} & \multicolumn{6}{c}{ \emph{effect size} = 0.25} \\
   \hline
15 & 0.92 & 0.93 & 0.94 & 0.96 & 0.92 & 0.93 & 0.89 & 0.88 & 0.82 & 0.85 & 0.86 & 0.89 \\ 
  30 & 0.89 & 0.91 & 0.94 & 0.93 & 0.86 & 0.85 & 0.82 & 0.79 & 0.89 & 0.92 & 0.92 & 0.93 \\ 
  50 & 0.86 & 0.91 & 0.94 & 0.91 & 0.90 & 0.87 & 0.78 & 0.77 & 0.87 & 0.92 & 0.94 & 0.95 \\ 
  70 & 0.81 & 0.89 & 0.91 & 0.92 & 0.86 & 0.83 & 0.69 & 0.73 & 0.85 & 0.93 & 0.96 & 0.96 \\ 
  90 & 0.80 & 0.89 & 0.91 & 0.92 & 0.89 & 0.86 & 0.69 & 0.75 & 0.87 & 0.94 & 0.97 & 0.99 \\ 
  120 & 0.78 & 0.89 & 0.93 & 0.92 & 0.85 & 0.85 & 0.65 & 0.77 & 0.92 & 0.98 & 0.99 & 0.99 \\ 
   \hline
\hline
& \multicolumn{6}{c|}{ \emph{effect size} =  0.45} & \multicolumn{6}{c}{ \emph{effect size} = 0.65} \\
   \hline
15 & 0.82 & 0.81 & 0.86 & 0.88 & 0.89 & 0.91 & 0.79 & 0.85 & 0.87 & 0.93 & 0.89 & 0.90 \\ 
  30 & 0.80 & 0.86 & 0.94 & 0.96 & 0.94 & 0.95 & 0.86 & 0.90 & 0.97 & 0.96 & 0.95 & 0.96 \\ 
  50 & 0.76 & 0.84 & 0.94 & 0.95 & 0.98 & 0.97 & 0.84 & 0.91 & 0.96 & 0.97 & 0.99 & 0.98 \\ 
  70 & 0.73 & 0.86 & 0.93 & 0.97 & 0.99 & 0.99 & 0.82 & 0.92 & 0.96 & 0.98 & 0.99 & 0.99 \\ 
  90 & 0.73 & 0.86 & 0.96 & 0.98 & 0.99 & 0.99 & 0.85 & 0.93 & 0.98 & 0.99 & 1.00 & 1.00 \\ 
  120 & 0.80 & 0.93 & 0.99 & 1.00 & 1.00 & 1.00 & 0.92 & 0.98 & 1.00 & 1.00 & 1.00 & 1.00 \\ 
   \hline
\end{tabular}
\end{table}

\begin{table}[ht]
\centering
\caption{Estimates of $dFDR$ for the BH procedure in \defref{dBH}  at level $\alpha = 0.1$, based on $500$ repetitions of  data  generated from the one-way ANOVA model  in \secref{numeric}, 
  for   combinations of $(m, n)$ and  $q_0/q \approx 0.25$; 
  the dFDR target is  $\frac{\alpha}{2}\left( 1 + q_0/q\right) \approx 0.0625$.  } 
\begin{tabular}{c|cccccc|cccccc}
  \hline
$ m \backslash n$ & 20 & 40 & 100 & 200 & 400 & 600 & 20 & 40 & 100 & 200 & 400 & 600 \\ 
  \hline
& \multicolumn{6}{c|}{ \emph{effect size} =  0.05} & \multicolumn{6}{c}{ \emph{effect size} = 0.25} \\
   \hline
15 & 0.06 & 0.04 & 0.04 & 0.02 & 0.03 & 0.02 & 0.04 & 0.02 & 0.03 & 0.02 & 0.02 & 0.02 \\ 
  30 & 0.06 & 0.05 & 0.03 & 0.02 & 0.03 & 0.03 & 0.03 & 0.03 & 0.03 & 0.03 & 0.03 & 0.03 \\ 
  50 & 0.09 & 0.05 & 0.03 & 0.03 & 0.02 & 0.03 & 0.05 & 0.04 & 0.03 & 0.03 & 0.03 & 0.03 \\ 
  70 & 0.12 & 0.06 & 0.04 & 0.03 & 0.03 & 0.03 & 0.06 & 0.04 & 0.03 & 0.03 & 0.03 & 0.03 \\ 
  90 & 0.12 & 0.07 & 0.04 & 0.03 & 0.02 & 0.03 & 0.06 & 0.04 & 0.03 & 0.03 & 0.03 & 0.03 \\ 
  120 & 0.12 & 0.06 & 0.03 & 0.03 & 0.03 & 0.03 & 0.06 & 0.04 & 0.03 & 0.03 & 0.03 & 0.03 \\ 
   \hline
\hline
& \multicolumn{6}{c|}{ \emph{effect size} =  0.45} & \multicolumn{6}{c}{ \emph{effect size} = 0.65} \\
   \hline
15 & 0.03 & 0.03 & 0.02 & 0.02 & 0.02 & 0.02 & 0.03 & 0.03 & 0.02 & 0.02 & 0.02 & 0.02 \\ 
  30 & 0.03 & 0.03 & 0.03 & 0.02 & 0.03 & 0.02 & 0.03 & 0.03 & 0.03 & 0.02 & 0.03 & 0.02 \\ 
  50 & 0.04 & 0.03 & 0.03 & 0.03 & 0.03 & 0.03 & 0.03 & 0.03 & 0.03 & 0.03 & 0.03 & 0.02 \\ 
  70 & 0.04 & 0.03 & 0.03 & 0.03 & 0.03 & 0.03 & 0.04 & 0.03 & 0.03 & 0.03 & 0.03 & 0.03 \\ 
  90 & 0.05 & 0.04 & 0.03 & 0.03 & 0.03 & 0.03 & 0.04 & 0.03 & 0.03 & 0.03 & 0.03 & 0.03 \\ 
  120 & 0.04 & 0.03 & 0.03 & 0.03 & 0.03 & 0.03 & 0.04 & 0.03 & 0.03 & 0.03 & 0.03 & 0.03 \\ 
   \hline
\end{tabular}
\end{table}

\begin{table}[ht]
\centering
\caption{Estimates of $P(dFDP \leq \frac{\alpha}{2}\left( 1 + q_0/q\right))$ for the BH procedure in \defref{dBH}  at level $\alpha = 0.1$, based on $500$ repetitions  data are generated from the one-way ANOVA model  in \secref{numeric}, 
  for   combinations of $(m, n)$ and  $q_0/q  \approx  0.5$.   } 
\begin{tabular}{c|cccccc|cccccc}
  \hline
$ m \backslash n$ & 20 & 40 & 100 & 200 & 400 & 600 & 20 & 40 & 100 & 200 & 400 & 600 \\ 
  \hline
& \multicolumn{6}{c|}{ \emph{effect size} =  0.05} & \multicolumn{6}{c}{ \emph{effect size} = 0.25} \\
   \hline
15 & 0.90 & 0.91 & 0.94 & 0.94 & 0.93 & 0.93 & 0.89 & 0.89 & 0.85 & 0.77 & 0.77 & 0.74 \\ 
  30 & 0.88 & 0.91 & 0.94 & 0.92 & 0.89 & 0.86 & 0.83 & 0.78 & 0.74 & 0.77 & 0.72 & 0.75 \\ 
  50 & 0.85 & 0.90 & 0.93 & 0.90 & 0.90 & 0.87 & 0.76 & 0.75 & 0.72 & 0.72 & 0.72 & 0.76 \\ 
  70 & 0.79 & 0.89 & 0.91 & 0.93 & 0.86 & 0.84 & 0.68 & 0.69 & 0.71 & 0.74 & 0.74 & 0.77 \\ 
  90 & 0.78 & 0.88 & 0.89 & 0.93 & 0.88 & 0.83 & 0.64 & 0.66 & 0.68 & 0.79 & 0.82 & 0.85 \\ 
  120 & 0.76 & 0.88 & 0.93 & 0.91 & 0.88 & 0.86 & 0.62 & 0.66 & 0.73 & 0.81 & 0.85 & 0.86 \\ 
   \hline
\hline
& \multicolumn{6}{c|}{ \emph{effect size} =  0.45} & \multicolumn{6}{c}{ \emph{effect size} = 0.65} \\
   \hline
15 & 0.84 & 0.79 & 0.77 & 0.75 & 0.76 & 0.73 & 0.79 & 0.73 & 0.76 & 0.76 & 0.75 & 0.73 \\ 
  30 & 0.69 & 0.68 & 0.75 & 0.78 & 0.75 & 0.77 & 0.67 & 0.70 & 0.77 & 0.78 & 0.76 & 0.78 \\ 
  50 & 0.60 & 0.68 & 0.75 & 0.76 & 0.77 & 0.80 & 0.62 & 0.70 & 0.76 & 0.79 & 0.79 & 0.80 \\ 
  70 & 0.54 & 0.65 & 0.73 & 0.78 & 0.78 & 0.80 & 0.60 & 0.70 & 0.76 & 0.82 & 0.80 & 0.82 \\ 
  90 & 0.55 & 0.67 & 0.75 & 0.83 & 0.85 & 0.88 & 0.61 & 0.72 & 0.79 & 0.85 & 0.85 & 0.88 \\ 
  120 & 0.54 & 0.70 & 0.82 & 0.87 & 0.88 & 0.90 & 0.64 & 0.76 & 0.85 & 0.88 & 0.89 & 0.91 \\ 
   \hline
\end{tabular}
\end{table}
\begin{table}[t!]
\centering
\caption{Estimates of $dFDR$ for the BH procedure in \defref{dBH}  at level $\alpha = 0.1$, based on $500$ repetitions of  data  generated from the one-way ANOVA model  in \secref{numeric}, 
  for   combinations of $(m, n)$ and  $q_0/q \approx 0.5$; 
  the dFDR target is  $\frac{\alpha}{2}\left( 1 + q_0/q\right) \approx 0.075$.  } 
\begin{tabular}{c|cccccc|cccccc}
  \hline
$ m \backslash n$ & 20 & 40 & 100 & 200 & 400 & 600 & 20 & 40 & 100 & 200 & 400 & 600 \\ 
  \hline
& \multicolumn{6}{c|}{ \emph{effect size} =  0.05} & \multicolumn{6}{c}{ \emph{effect size} = 0.25} \\
   \hline
15 & 0.08 & 0.07 & 0.04 & 0.04 & 0.05 & 0.04 & 0.06 & 0.05 & 0.04 & 0.04 & 0.04 & 0.04 \\ 
  30 & 0.08 & 0.07 & 0.04 & 0.05 & 0.05 & 0.05 & 0.07 & 0.05 & 0.05 & 0.05 & 0.05 & 0.05 \\ 
  50 & 0.11 & 0.08 & 0.05 & 0.06 & 0.04 & 0.04 & 0.08 & 0.06 & 0.05 & 0.05 & 0.05 & 0.05 \\ 
  70 & 0.16 & 0.09 & 0.06 & 0.04 & 0.05 & 0.05 & 0.11 & 0.07 & 0.05 & 0.05 & 0.06 & 0.05 \\ 
  90 & 0.17 & 0.09 & 0.07 & 0.04 & 0.03 & 0.04 & 0.11 & 0.07 & 0.06 & 0.05 & 0.05 & 0.05 \\ 
  120 & 0.16 & 0.08 & 0.05 & 0.04 & 0.04 & 0.04 & 0.10 & 0.07 & 0.06 & 0.05 & 0.05 & 0.05 \\ 
   \hline
\hline
& \multicolumn{6}{c|}{ \emph{effect size} =  0.45} & \multicolumn{6}{c}{ \emph{effect size} = 0.65} \\
   \hline
15 & 0.06 & 0.05 & 0.04 & 0.04 & 0.04 & 0.04 & 0.05 & 0.05 & 0.04 & 0.04 & 0.04 & 0.04 \\ 
  30 & 0.06 & 0.05 & 0.05 & 0.05 & 0.05 & 0.05 & 0.06 & 0.05 & 0.05 & 0.05 & 0.05 & 0.05 \\ 
  50 & 0.08 & 0.06 & 0.05 & 0.05 & 0.05 & 0.05 & 0.07 & 0.06 & 0.05 & 0.05 & 0.05 & 0.05 \\ 
  70 & 0.08 & 0.06 & 0.05 & 0.05 & 0.05 & 0.05 & 0.07 & 0.06 & 0.05 & 0.05 & 0.05 & 0.05 \\ 
  90 & 0.08 & 0.06 & 0.05 & 0.05 & 0.05 & 0.05 & 0.07 & 0.06 & 0.05 & 0.05 & 0.05 & 0.05 \\ 
  120 & 0.08 & 0.06 & 0.05 & 0.05 & 0.05 & 0.05 & 0.07 & 0.06 & 0.05 & 0.05 & 0.05 & 0.05 \\ 
   \hline
\end{tabular}
\
\end{table}

  \newpage
  \ \
  \newpage
  \ \
  \newpage
    \ \
  \newpage
  \ \
    \newpage
  \ \
  \newpage
  \ \
  \newpage

\bibliographystyle{spbasic}

\bibliography{TukeyConj_bib}

\end{document}